\documentclass[a4paper, 10pt]{amsart}
\usepackage{amsmath, amssymb, amsthm, bm}	%ams environment
\usepackage{eucal}
\usepackage{braket, stmaryrd,mathrsfs}
\usepackage{cite}
\usepackage[all]{xy}
\usepackage{youngtab}					%draw Young diagrams
\usepackage[dvipdfmx]{graphicx,xcolor}
\usepackage{tikz}

\calclayout

\newcommand{\pr}[1]{#1^{\prime}}

\newcommand{\del}{\partial}

\newcommand{\mfrak}[1]{\mathfrak{#1}}
\newcommand{\mcal}[1]{\mathcal{#1}}
\newcommand{\mbb}[1]{\mathbb{#1}}
\newcommand{\mrm}[1]{\mathrm{#1}}

\newcommand{\what}[1]{\widehat{#1}}
\newcommand{\no}[1]{:\hspace{-3pt} #1\hspace{-3pt}:\hspace{3pt}}

\makeatletter
\newcommand{\xRightarrow}[2][]{\ext@arrow 0359\Rightarrowfill@{#1}{#2}}
\makeatother

\theoremstyle{plain}
\newtheorem{thm}{Theorem}[section]
\newtheorem{lem}[thm]{Lemma}
\newtheorem{prop}[thm]{Proposition}
\newtheorem{cor}[thm]{Corollary}
\theoremstyle{definition}
\newtheorem{defn}[thm]{Definition}%[section]
\theoremstyle{remark}
\newtheorem{rem}[thm]{Remark}%[section]

\makeatletter
    
    \@addtoreset{equation}{section}
\makeatother

\title{Free field approach to the Macdonald process}
\date{\today}
\author{Shinji Koshida}
\address{Department of Physics, Faculty of Science and Engineering, Chuo University, Kasuga, Bunkyo, Tokyo 112-8551, Japan}
\email{koshida@phys.chuo-u.ac.jp}

\begin{document}

\begin{abstract}
The Macdonald process is a stochastic process on the collection of partitions
that is a $(q,t)$-deformed generalization of the Schur process.
In this paper, we approach the Macdonald process identifying the space of symmetric functions
with a Fock representation of a Heisenberg algebra.
By using the free field realization of operators diagonalized by the Macdonald symmetric functions,
we propose a method of computing several correlation functions with respect to the Macdonald process.
It is well-known that expectation value of several observables for the Macdonald process admit determinantal expression.
We find that this determinantal structure is apparent in free field realization of the corresponding operators and, furthermore, it has a natural interpretation in the language of free fermions at the Schur limit.
We also propose a generalized Macdonald measure motivated by recent studies on generalized Macdonald functions
whose existence relies on the Hopf algebra structure of the Ding--Iohara--Miki algebra.
\end{abstract}

\subjclass[2010]{05E05,33D52,16T05}
\keywords{Macdonald process, Macdonald symmetric function, Ding--Iohara--Miki algebra, Generalized Macdonald functions, Generalized Macdonald measure}

\maketitle

\section{Introduction}
\subsection{Backgrounds}
Let $\mbb{Y}_{n}$ be the collection of partitions of $n\in\mbb{Z}_{\ge 1}$,
and set $\mbb{Y}:=\bigcup_{n=0}^{\infty}\mbb{Y}_{n}$, where $\mbb{Y}_{0}=\{\emptyset\}$.
The Macdonald measure $\mbb{MM}_{q,t}$ is a probability measure on $\mbb{Y}$ defined by \cite{BorodinCorwin2014,BorodinCorwinGorinShakirov2016}
\begin{equation*}
	\mbb{MM}_{q,t}(\lambda)=\frac{1}{\Pi(X,Y;q,t)}P_{\lambda}(X;q,t)Q_{\lambda}(Y;q,t),\ \ \lambda\in\mbb{Y}.
\end{equation*}
Here $P_{\lambda}(X;q,t)$ is the Macdonald symmetric function of $X=(x_{1},x_{2},\dots)$ for a partition $\lambda$
and $Q_{\lambda}(Y;q,t)$ is its dual symmetric function of $Y=(y_{1},y_{2},\dots)$ (see Section \ref{sect:symm_funcs} for definition).
From the Cauchy-type identity, the normalization factor is computed as
\begin{equation*}
	\Pi (X,Y;q,t)=\sum_{\lambda\in\mbb{Y}}P_{\lambda}(X;q,t)Q_{\lambda}(Y;q,t)=\prod_{i,j\ge 1}\frac{(tx_{i}y_{j};q)_{\infty}}{(x_{i}y_{j};q)_{\infty}},
\end{equation*}
where $(a;q)_{\infty}=\prod_{n=0}^{\infty}(1-aq^{n})$.
In the following, we suppress the parameters $q$ and $t$ if there is no ambiguity, for instance, by writing $P_{\lambda}(X)=P_{\lambda}(X;q,t)$.
Precisely speaking, to obtain a genuine probability measure,
we have to adopt a nonnegative specialization of the Macdonald symmetric functions whose classification was conjectured in \cite{Kerov1992} and recently proved in \cite{Matveev2019}.

As a generalization of the Macdonald measure, the $N$-step Macdonald process \cite{BorodinCorwin2014,BorodinCorwinGorinShakirov2016} for $N\ge 1$
is a probability measure $\mbb{MP}^{N}_{q,t}$ on $\mbb{Y}^{N}$ defined so that
the probability for a sequence $(\lambda^{(1)},\dots,\lambda^{(N)})\in \mbb{Y}^{N}$ of partitions is given by
\begin{align}
\label{eq:Macdonald_process}
	&\mbb{MP}^{N}_{q,t}(\lambda^{(1)},\dots,\lambda^{(N)})\\
	&:=\frac{P_{\lambda^{(1)}}(X^{(1)})\Psi_{\lambda^{(1)},\lambda^{(2)}}(Y^{(1)},X^{(2)})\cdots \Psi_{\lambda^{(N-1)},\lambda^{(N)}}(Y^{(N-1)},X^{(N)})Q_{\lambda^{(N)}}(Y^{(N)})}{\prod_{1\le i\le j\le N}\Pi(X^{(i)},Y^{(j)})}.\notag
\end{align}
Here the transition function $\Psi_{\lambda,\mu}(Y,X)$ is given by
\begin{equation}
\label{eq:transition_function}
	\Psi_{\lambda,\mu}(Y,X)=\sum_{\nu\in\mbb{Y}}Q_{\lambda/\nu}(Y)P_{\mu/\nu}(X),\ \ \lambda,\mu\in\mbb{Y},
\end{equation}
with $P_{\lambda/\nu}$ being the Macdonald symmetric functions for a skew-partition $\lambda/\nu$ and $Q_{\lambda/\mu}$ being its dual.
The case of $N=1$ is just the Macdonald measure, $\mbb{MM}_{q,t}=\mbb{MP}^{1}_{q,t}$.

It is known that the Macdonald process reduces to several interesting stochastic models by specializing the variables and limiting the parameters
and has given many applications to probability theory. Examples include the $q$-TASEP \cite{BorodinCorwin2014}, general $\beta$-ensembles \cite{BorodinGorin2015}, Hall--Littlewood plane partitions \cite{Dimitrov2018}, Whittaker processes \cite{Oconnell2012, OconnellSeppalainenZygouras2014,CorwinOconnellSeppalainenZygouras2014}, Kingman partition structures \cite{Petrov2009} (see also \cite{BorodinGorin2012,Borodin2014,BorodinPetrov2014,Corwin2014}).
In particular, when we set $q=t$, the Macdonald symmetric functions reduce to the Schur functions
and, correspondingly, the Macdonald process reduces to the Schur process \cite{Okounkov2001,OkounkovReshetikhin2003}.
The Schur process can be shown to be a determinantal point process (DPP) in a simple manner
owing to the infinite-wedge realization of the Schur functions and action of free fermions \cite{Okounkov2001}.
The analogous field theoretical approach to the Macdonald process is, however, absent to the author's knowledge.

\subsection{Correspondence among correlation functions}
In the present paper, we study the Macdonald process by identifying the space of symmetric functions with a Fock representation of a Heisenberg algebra as was suggested in \cite{FodaWu2017}. As its first step, we express a correlation function of the Macdonald process in terms of matrix elements of operators on the Fock space.
We set $\mbb{F}=\mbb{C}(q,t)$.
Then, we regard a function $f:\mbb{Y}\to\mbb{F}$ as a random variable or an observable.
\begin{defn}
Let $f_{1},\dots, f_{N}:\mbb{Y}\to\mbb{F}$ be random variables.
The correlation function $\mbb{E}^{N}_{q,t}[f_{1}[1]\cdots f_{N}[N]]$ with respect to the $N$-step Macdonald process
is defined by
\begin{equation*}
	\mbb{E}^{N}_{q,t}[f_{1}[1]\cdots f_{N}[N]]
	:=\sum_{(\lambda^{(1)},\dots,\lambda^{(N)})\in\mbb{Y}^{N}}f_{1}(\lambda^{(1)})\cdots f_{N}(\lambda^{(N)})\mbb{MP}_{q,t}^{N}(\lambda^{(1)},\dots,\lambda^{(N)}).
\end{equation*}
In the case of $N=1$, we simply write $\mbb{E}_{q,t}[f]:=\mbb{E}^{1}_{q,t}[f[1]]$.
\end{defn}

We write $\Lambda$ for the ring of symmetric functions over $\mbb{F}$.
Then it is isomorphic to a Fock representation $\mcal{F}$ and its dual $\mcal{F}^{\dagger}$ of a Heisenberg algebra (see Section \ref{sect:free_field_realization}),
where the Macdonald symmetric function $P_{\lambda}$ corresponding to $\lambda\in\mbb{Y}$ is
identified with $\ket{P_{\lambda}}\in\mcal{F}$ and its dual $Q_{\lambda}$ is identified with $\bra{Q_{\lambda}}\in\mcal{F}^{\dagger}$.
We introduce operators on $\mcal{F}$:
\begin{equation}
\label{eq:Gamma}
	\Gamma (X)_{\pm}=\exp\left(\sum_{n>0}\frac{1-t^{n}}{1-q^{n}}\frac{p_{n}(X)}{n}a_{\pm n}\right),
\end{equation}
where $a_{n}$, $n\in\mbb{Z}\backslash\{0\}$ are generators of the Heisenberg algebra
and $p_{n}(X)$, $n\ge 1$ is the $n$-th power-sum symmetric function of variables $X=(x_{1},x_{2},\dots)$ (see Section \ref{sect:symm_funcs}).

We write $\mbb{F}[\mbb{Y}]:=\{f:\mbb{Y}\to\mbb{F}\}$ for the set of random variables.
The method of computing correlation functions using operators (difference operators in typical cases) that are diagonalized by the Macdonald polynomials
has been developed in \cite{BorodinCorwin2014,BorodinCorwin2015, BorodinGorin2015, BorodinCorwinGorinShakirov2016, Dimitrov2018,GorinZhang2018}.
Here we shall formulate a complementary algebraic scheme to compute correlation functions.
 
\begin{defn}
Regarding the values of a random variable as the eigenvalues of an operator, we define a mapping
\begin{equation*}
	\mcal{O}:\mbb{F}[\mbb{Y}]\to \mrm{End}(\mcal{F});\ \ f\mapsto \sum_{\lambda\in\mbb{Y}}f(\lambda)\ket{P_{\lambda}}\bra{Q_{\lambda}}.
\end{equation*}
For a random variable $f\in \mbb{F}[\mbb{Y}]$, we also define
\begin{equation*}
	\psi_{f}^{X,Y}:=\Gamma (Y)_{+}\mcal{O}(f)\Gamma (X)_{-}.
\end{equation*}
\end{defn}

Then we have the following correspondence between correlation functions under the Macdonald process and matrix elements in the Fock space.
\begin{thm}
\label{thm:correlation_operator}
Let $f_{1},\dots, f_{N}\in \mbb{F}[\mbb{Y}]$ be random variables.
Then their correlation function with respect to the $N$-step Macdonald process becomes
\begin{equation}
\label{eq:correlation_matrix_elem}
	\mbb{E}_{q,t}^{N}[f_{1}[1]\cdots f_{N}[N]]
	=\frac{\braket{0|\psi_{f_{N}}^{X^{(N)},Y^{(N)}}\cdots \psi_{f_{1}}^{X^{(1)},Y^{(1)}}|0}} {\braket{0|\psi_{1}^{X^{(N)},Y^{(N)}}\cdots \psi_{1}^{X^{(1)},Y^{(1)}}|0}}.
\end{equation}
Here $\ket{0}\in\mcal{F}$ and $\bra{0}\in \mcal{F}^{\dagger}$ are the vacuum vectors
and $1\in \mbb{F}[\mbb{Y}]$ is the unit constant function.
\end{thm}

Theorem \ref{thm:correlation_operator} is proved in Section \ref{sect:free_field_realization}.

\subsection{Determinantal expression of operators}
Due to Theorem \ref{thm:correlation_operator}, the problem reduces to how efficiently we can compute the matrix elements in (\ref{eq:correlation_matrix_elem}).
We use the free field realization of operators that are diagonalized by the Macdonald symmetric functions due to \cite{FeiginHashizumeHoshinoShiraishiYanagida2009}
to make the computation of (\ref{eq:correlation_matrix_elem}) possible.

It is well-known \cite{BorodinCorwin2014,BorodinCorwinGorinShakirov2016} that several expectation values concerning the Macdonald process admit determinantal expression. The determinantal structure of the Macdonald processes is a long-standing mystery as it is not a DPP, and the initial motivation of this work was to understand the origin of this determinantal structure. We found that the determinantal structure gets apparent in the free field realization of operators
diagonalized by the Macdonald symmetric functions as we are overviewing below.

The Macdonald symmetric functions are simultaneous eigenfunctions of commuting operators including the Macdonald operators.
Under the isomorphism $\mcal{F}\simeq \Lambda$, these operators are identified with operators on $\mcal{F}$,
which were studied in \cite{Shiraishi2006,FeiginHashizumeHoshinoShiraishiYanagida2009} as the free field realization.
We seek different expression of these free field realizations involving determinant.
Let us introduce a {\it vertex operator}
\begin{equation}
\label{eq:vertex_eta}
	\eta (z)=\exp\left(\sum_{n>0}\frac{1-t^{-n}}{n}a_{-n}z^{n}\right)\exp\left(-\sum_{n>0}\frac{1-t^{n}}{n}a_{n}z^{-n}\right),
\end{equation}
which lies in $\mrm{End}(\mcal{F})[[z,z^{-1}]]$.
\begin{thm}
\label{thm:Macdonald_determinant}
Let $r=1,2,\dots$.
The free field realization $\what{E}_{r}$ of the $r$-th Macdonald operator is expressed as
\begin{equation}
\label{eq:Macdonald_free_field_determinant}
	\what{E}_{r}=\frac{t^{-r}}{r!}\int\left(\prod_{i=1}^{r}\frac{dz_{i}}{2\pi\sqrt{-1}}\right)\det\left(\frac{1}{z_{i}-t^{-1}z_{j}}\right)_{1\le i,j\le r}
				\no{\eta(z_{1})\cdots \eta(z_{r})}.
\end{equation}
\end{thm}

Throughout this paper, we understand the integral $\int \frac{dz}{2\pi\sqrt{-1}}$ as the functional taking the coefficient of $z^{-1}$ and a rational function of formal variables like $\frac{1}{z-\gamma w}$, $\gamma\in\mbb{F}$ as a formal series expanded in $\mbb{F}[w]((z^{-1}))$. In particular, $\frac{1}{z-\gamma w}\neq -\frac{1}{\gamma w-z}$.
We prove Theorem \ref{thm:Macdonald_determinant} in Section \ref{sect:determinant} and also show determinantal expression of other operators.
Combining these determinantal expressions with Theorem \ref{thm:correlation_operator}, we derive determinantal formulas of correlation functions in Section \ref{sect:expectation}. Hence, we can say that the determinant appearing as a integrand in (\ref{eq:Macdonald_free_field_determinant}) gives determinantal formulas. 

To understand still more fundamental origins of the determinantal structure, we also consider the Schur-limit in Section \ref{sect:determinant}. Consequently, we find that the operator (\ref{eq:Macdonald_free_field_determinant}) reduces to action of fermionic operators that give the determinantal integrand in (\ref{eq:Macdonald_free_field_determinant}). 

We remark that our approach is fully algebraic and formal and does not require any specialization of the variables.
Therefore, our results apply to any models obtained by specialization of the Macdonald process.

\subsection{Generalized Macdonald measure}
We also propose a certain generalization of the Macdonald measure.
It is known that the Ding--Iohara--Miki (DIM) algebra \cite{DingIohara1997,Miki2007} plays a relevant role in the theory of the Macdonald symmetric functions.
In \cite{AwataFeiginHoshinoKanaiShiraishiYanagida2011}, the authors proposed a family of generalized Macdonald functions using the coproduct structure of the DIM algebra. We consider a generalization of the Macdonald measure replacing the Macdonald symmetric functions by generalized Macdonald functions as follows.

Let $m\in\mbb{N}$ be fixed.
In \cite{FeiginHoshinoShibaharaShiraishiYanagida2010,AwataFeiginHoshinoKanaiShiraishiYanagida2011}, it was proved that the $m$-fold tensor product
$\widetilde{\mcal{F}}^{\otimes m}$ of $\widetilde{\mcal{F}}:=\mbb{C}(q^{1/4},t^{1/4})\otimes_{\mbb{F}}\mcal{F}$
admits a Macdonald type basis labeled by $m$-tuple of partitions.
Under the isomorphism $\mcal{F}^{\otimes m}\simeq \Lambda^{\otimes m}$,
the level $m$ generalized Macdonald functions $P_{\bm{\lambda}}(\bm{X})$, $\bm{\lambda}=(\lambda^{(1)},\dots,\lambda^{(m)})\in\mbb{Y}^{m}$, $\bm{X}=(X^{(1)},\dots, X^{(m)})$,
and their dual functions $Q_{\bm{\lambda}}(\bm{X})$ are defined (see Proposition \ref{prop:char_generalized_Macdonald} and Definition \ref{def:generalized_Macdonald} below).
In Section \ref{sect:gen_Mac_meas}, we will define the level $m$ generalized Macdonald measure as a probability measure $\mbb{GM}^{m}_{q,t}$ on $\mbb{Y}^{m}$ so that
\begin{equation*}
	\mbb{GM}^{m}_{q,t}(\bm{\lambda})\propto P_{\bm{\lambda}}(\bm{X})Q_{\bm{\lambda}}(\bm{Y}),\ \ \bm{\lambda}\in\mbb{Y}^{m}
\end{equation*}
and write $\mbb{GE}^{m}_{q,t}$ for the expectation value under $\mbb{GM}^{m}_{q,t}$.

As a demonstration, we will compute the expectation value of a random variable $\what{\mcal{E}}_{1}^{(m)}$, defined by
\begin{equation*}
	\what{\mcal{E}}^{(m)}_{1}(\bm{\lambda}):=\sum_{j=1}^{m}\left(1+(t-1)\sum_{i\ge 1}(q^{\lambda^{(j)}_{i}}-1)t^{-i}\right),\ \ \bm{\lambda}=(\lambda^{(1)},\dots, \lambda^{(m)})\in\mbb{Y}^{m}.
\end{equation*}

\begin{thm}
\label{thm:generalized_Mcdonald_expectation}
Set
\begin{align*}
	H(w;X)&:=\prod_{i\ge 1}\frac{1-tx_{i}w}{1-x_{i}w}, \\
	M(z;X)&:=\prod_{k\ge 1}\frac{(1-zx_{k})(1-q^{-1}zx_{k})}{(1-q^{-1}tzx_{k})(1-t^{-1}zx_{k})},\ \ X=(x_{1},x_{2},\dots),
\end{align*}
and $p=q/t$.
Then, we have
\begin{align*}
	\mbb{GE}^{m}_{q,t}[\what{\mcal{E}}_{1}^{(m)}]=\int\frac{dz}{2\pi\sqrt{-1}z}\sum_{i=1}^{m}\frac{\prod_{j=1}^{i-1}M(p^{-(j+1)/2}z;Y^{(i)})}{H(p^{(i-1)/2}z^{-1};X^{(i)})H(t^{-1}p^{-(i-1)/2}z;Y^{(i)})}.
\end{align*}
\end{thm}

\subsection{Future directions}

We close this introduction by make comments on future directions.
\subsubsection{Application to stochastic models}
Since our results are formal and do not require any specialization of variables,
they apply to any reduction of the Macdonald process.
For application to stochastic models, however, one has to specialize variables and carry out further analyses
typically studying asymptotic behaviors (e.g. \cite{BorodinCorwinRemenik2013,BorodinCorwin2014,BorodinCorwin2014b, BorodinCorwinFerrari2014,Barraquand2015, BorodinGorin2015, FerrariVeto2015}).
It is not clear so far how our results are useful for such application and more study is needed.

\subsubsection{Further studies on generalized Macdonald measure}
We need to study generalized Macdonald measure in application to stochastic models.
For this purpose, we have to consider positive specialization of generalized Macdonald functions to define a genuine probability measure.
We also need to define {\it skew} generalized Macdonald functions and combinatorial formula of their few-variable specialization.
To all these aims, the first step is to study the Pieri-type formulas for generalized Macdonald functions.

\subsubsection{Elliptic generalization}
In \cite{Saito2013, Saito2014}, the elliptic Macdonald operators were realized as operators on a Fock space
by means of the elliptic DIM algebra.
Though the elliptic Macdonald symmetric functions as a basis of the Fock space have not been captured so far,
once a Macdonald-type basis is found, a similar story as in this paper would work in the elliptic case.

\subsubsection{Relation to higher spin six-vertex models}
Another pillar than the Macdonald process in the field of integrable probability is a higher spin six-vertex model and its variants \cite{BorodinCorwinGorin2016,CorwinPetrov2016,Borodin2017,BorodinPetrov2018,BorodinWheeler2018,BorodinWheeler2019}
and there are attempts to understand these two on the same footing \cite{GarbalideGierWheeler2017,Borodin2018,BorodinWheeler2018, BorodinWheeler2019}.
Notably, partition functions of a higher spin six-vertex model give a family of symmetric rational functions \cite{Borodin2017}
that are regarded as generalization of the Hall--Littlewood polynomials.
It is also known \cite{BrubakerBuciumasBumpFriedberg2018, BrubakerBuciumasBumpGustafsson2018} that,
for some lattice models such as a metaplectic ice model,
a vertex operator acting on a Fock space works as a transfer matrix
and its matrix elements give a family of symmetric functions.
Since a Fock space and vertex operators are also basic tools in this paper,
the present work could give a new insight to this subject from a perspective of
the representation theory of quantum algebras.

\vspace{10pt}
The present paper is organized as follows:
In Section \ref{sect:symm_funcs}, we review the basic notions of symmetric functions and introduce the Macdonald symmetric functions, which are needed to define the Macdonald process (\ref{eq:Macdonald_process}).
In Section \ref{sect:free_field_realization}, we recall that the space of symmetric functions is isomorphic to a Fock representation of a Heisenberg algebra and prove Theorem \ref{thm:correlation_operator}. We also see the free field realization of operators that are diagonalized by the Macdonald symmetric functions.
In Section \ref{sect:determinant}, we rewrite the free field realizations in Section \ref{sect:free_field_realization} by using determinants
to prove Theorem \ref{thm:Macdonald_determinant} and its analogues to other operators. We also consider the Schur-limit to better understand the origin of the determinantal structure.
In Section \ref{sect:expectation}, we present applications of Theorem \ref{thm:correlation_operator} and the results in Section \ref{sect:determinant} to compute correlation functions of some observables.
In Section \ref{sect:gen_Mac_meas}, after overviewing the theory of the DIM algebra, we introduce a generalized Macdonald measure and prove Theorem \ref{thm:generalized_Mcdonald_expectation}.
In Appendix \ref{app:proof_free_field_G_inverse}, we give a proof of the free field realization of a certain family of operators diagonalized by the Macdonald symmetric functions.

\vspace{10pt}
Throughout this paper, we use the notations
\begin{equation*}
	[n]_{q}=\frac{1-q^{n}}{1-q},\ \ [n]_{q}!=\prod_{k=1}^{n}[k]_{q},\ \ (x;q)_{n}=\prod_{k=0}^{n-1}(1-xq^{k}).
\end{equation*}

\subsection*{Acknowledgements}
The author is grateful to Makoto Katori, Tomohiro Sasamoto, Takashi Imamura, Yoshihiro Takeyama and Alexander Bufetov for fruitful discussion.
He also thanks anonymous referees for helping him improve the manuscript with convenient advice and suggestions.
This work was supported by the Grant-in-Aid for JSPS Fellows (No.\ 17J09658, No.\ 19J01279)

\section{Preliminaries on symmetric functions}
\label{sect:symm_funcs}
In this paper, we regard the parameters $q$ and $t$ as indeterminates unless otherwise specified and set $\mbb{F}:=\mbb{C}(q,t)$.
Let us prepare some terminologies of symmetric functions and introduce the Macdonald symmetric functions.
The relevant reference is \cite{Macdonald1995}.

\subsection{Ring of symmetric functions}
Let $\Lambda^{(n)}=\mbb{F}[x_{1},\dots, x_{n}]^{\mfrak{S}_{n}}$ be the ring of symmetric polynomials in $n$ variables over $\mbb{F}$.
The ring of symmetric functions is defined as the projective limit $\Lambda=\varprojlim_{n}\Lambda^{(n)}$ in the category of graded rings, where, given $m>n$, the projection $\Lambda^{(m)}\to \Lambda^{(n)}$ sends the last $m-n$ variables to zero.
For a symmetric function $F\in\Lambda$, its image under the canonical surjection $\Lambda\twoheadrightarrow \Lambda^{(n)}$, $n\in\mbb{Z}_{\ge 0}$ will be denoted as $F^{(n)}$, and call it the $n$-variable reduction of $F$.
In the following, we write $X=(x_{1},x_{2},\dots)$ for a set of infinitely many variables
and use the notation $\Lambda_{X}$ if the variables are need to be specified.

For a partition $\lambda\in \mbb{Y}_{n}$ of $n\in\mbb{Z}_{\ge 1}$, we write $|\lambda|=n$ for its weight.
Its length is defined by $\ell(\lambda):=\mrm{max}\{i=1,2,\dots |\lambda_{i}>0\}$,
and the multiplicity of $i\in\mbb{Z}_{\ge 1}$ in $\lambda$ is defined by $m_{i}(\lambda):=|\{j=1,2,\dots|\lambda_{j}=i\}|$.
In terms of the multiplicity, we also express a partition as $\lambda=(1^{m_{1}(\lambda)}2^{m_{2}(\lambda)}\cdots)$.
For two partitions $\lambda$, $\mu\in\mbb{Y}$, we write $\lambda\ge\mu$ if $|\lambda|=|\mu|$ and
$\lambda_{1}+\cdots+\lambda_{n}\ge \mu_{1}+\cdots+\mu_{n}$, $n=1,2,\dots$.
Then $\ge$ defines a partial order on $\mbb{Y}$ called the dominance order.
For two partitions $\lambda$, $\mu\in\mbb{Y}$, we say that $\mu$ is included in $\lambda$ if
$\lambda_{i}\ge \mu_{i}$, $i=1,2,\dots$, hold, and write $\mu\subset\lambda$.
Their difference is called a skew-partition and denoted as $\lambda/\mu$.

Let us introduce some important symmetric functions.
For $r\in\mbb{Z}_{>0}$, the $r$-th elementary symmetric function $e_{r}(X)$ is defined by
$e_{r}(X):=\sum_{i_{1}<\cdots <i_{r}}x_{i_{1}}\cdots x_{i_{r}}$
and the $r$-th power-sum symmetric function $p_{r}(X)$ is defined by
$p_{r}(X):=\sum_{i\ge 1}x_{i}^{r}$.
For a partition $\lambda=(\lambda_{1},\lambda_{2},\dots)\in\mbb{Y}$, we also define
$p_{\lambda}(X):=p_{\lambda_{1}}(X)p_{\lambda_{2}}(X)\cdots$.
Other important symmetric functions are monomial symmetric functions.
Let $\lambda\in\mbb{Y}$ be a partition and $n$ be an integer larger than or equal to $\ell(\lambda)$.
We may regard $\lambda=(\lambda_{1},\dots,\lambda_{n})$ as an element in $(\mbb{Z}_{\ge 0})^{n}$, on which the $n$-th symmetric group acts by permutation of components. A monomial symmetric polynomial of $n$ variables is defined by $m^{(n)}_{\lambda}(x_{1},\dots, x_{n})=\sum_{\sigma}x_{1}^{\lambda_{\sigma (1)}}\cdots x_{n}^{\lambda_{\sigma (n)}}$, where the sum runs over distinct terms. Then the collection $\set{m_{\lambda}^{(n)}(x_{1},\dots, x_{n}):n\ge \ell (\lambda)}$ determines a unique symmetric function $m_{\lambda}(X)\in\Lambda$ called the monomial symmetric function corresponding to $\lambda$.
It is known that the collections $\set{p_{\lambda}}_{\lambda\in\mbb{Y}}$ and $\set{m_{\lambda}}_{\lambda\in\mbb{Y}}$ form $\mbb{F}$-bases of $\Lambda$.

An algebraic homomorphism $\rho:\Lambda\to \mbb{F}$ is called a specialization. We often write the image of $F\in\Lambda$ under a specialization $\rho$ as $F(\rho)$ instead of $\rho (F)$. We frequently consider specializations associated to partitions. For $\lambda\in\mbb{Y}$ and $n\in\mbb{Z}$, we define a specialization $q^{\lambda}t^{-\delta+n}:\Lambda\to\mbb{F}$ by
\begin{align*}
	p_{r}(q^{\lambda}t^{-\delta+n})&:=\sum_{i=1}^{\ell(\lambda)}(q^{\lambda_{i}}t^{-i+n})^{r}+\frac{t^{-r(\ell(\lambda)+1-n)}}{1-t^{-r}},& r&=1,2,\dots,
\end{align*}
which is interpreted as substitution $x_{i}\mapsto q^{\lambda_{i}}t^{-i+n}$, $i\ge 1$.
A specialization $q^{-\lambda}t^{\delta-n}$ is defined just by replacing $q$ by $q^{-1}$ and $t$ by $t^{-1}$ in the above formula.

\subsection{Macdonald symmetric functions}
\label{subsect:intro_macdonald}
To define the Macdonald symmetric functions, we introduce the Macdonald difference operators.
Fix $n\in\mbb{Z}_{\ge 1}$.
Then for $r=1,\dots, n$, the $r$-th Macdonald difference operator $D^{(n)}_{r}$ acting on $\Lambda^{(n)}$ is defined by \cite[Section VI. 3]{Macdonald1995}
\begin{equation*}
	D^{(n)}_{r}=D^{(n)}_{r}(q,t):=t^{r(r-1)/2}\sum_{\substack{I\subset \{1,2,\dots, n\}\\|I|=r}}\prod_{\substack{i\in I\\j\not\in I}}\frac{tx_{i}-x_{j}}{x_{i}-x_{j}}\prod_{i\in I}T_{q,x_{i}},
\end{equation*}
where $T_{q,x_{i}}$ is the $q$-shift operator $(T_{q,x_{i}}f)(x_{1},\dots,x_{n}):=f(x_{1},\dots, qx_{i},\dots, x_{n})$.

For a partition $\lambda\in\mbb{Y}$, the corresponding Macdonald symmetric function $P_{\lambda}(X;q,t)\in \Lambda_{X}$ is uniquely characterized by the triangularity
\begin{equation}
\label{eq:triangularity}
	P_{\lambda}(X;q,t)=m_{\lambda}(X)+\sum_{\mu;\mu<\lambda}c_{\lambda\mu}(q,t)m_{\mu}(X),\ \  c_{\lambda\mu}(q,t)\in\mbb{F}
\end{equation}
and the property that, for each $n\ge \ell(\lambda)$, the $n$-variable reduction $P_{\lambda}^{(n)}(x_{1},\dots, x_{n};q,t)$ is a simultaneous eigenfunction of the Macdonald difference operators so that
\begin{equation*}
	D^{(n)}_{r}P^{(n)}_{\lambda}(x_{1},\dots,x_{n};q,t)=e^{(n)}_{r}(q^{\lambda_{1}}t^{n-1},\dots, q^{\lambda_{n}})P^{(n)}_{\lambda}(x_{1},\dots,x_{n};q,t)\end{equation*}
for each $r=1,\dots, n$. Note that the triangularity property ensures that the Macdonald symmetric functions $P_{\lambda}(X;q,t)$, $\lambda\in\mbb{Y}$ form a basis of $\Lambda$.

Though the Macdonald difference operators themselves do not extend to operators on $\Lambda$, it is known \cite[Proposition 3.3]{FeiginHashizumeHoshinoShiraishiYanagida2009} (see also \cite[Chapter VI, Section 4]{Macdonald1995}) that we may define ones that extend to $\Lambda$.
For each $r\in\mbb{Z}_{\ge 1}$ and $n\ge r$, we consider a difference operator on $\Lambda^{(n)}$
\begin{equation*}
	E^{(n)}_{r}:=\sum_{k=0}^{r}\frac{t^{-nr-\binom{r-k+1}{2}}}{(t^{-1};t^{-1})_{r-k}}D^{(n)}_{k},
\end{equation*}
with the convention that $D^{(n)}_{0}=1$.
Then the family $\set{E^{(n)}_{r}:n\ge r}$ gives a unique operator $E_{r}=E_{r}(q,t)=\varprojlim_{n}E^{(n)}_{r}$ that is diagonalized by the Macdonald symmetric functions so that
\begin{equation*}
	E_{r}P_{\lambda}(X;q,t)=e_{r}(q^{\lambda}t^{-\delta})P_{\lambda}(X;q,t),\ \ \lambda\in\mbb{Y}.
\end{equation*}

We next introduce the Macdonald symmetric functions for skew-partitions.
Let $X=(x_{1},x_{2},\dots)$ and $Y=(y_{1},y_{2},\dots)$ be two sets of variables
and suppose that they are combined to be a single set of variables $(X,Y)=(x_{1},x_{2},\dots,y_{1},y_{2},\dots)$.
Then we can think of a Macdonald symmetric function $P_{\lambda}(X,Y)\in\Lambda_{(X,Y)}$ of these variables.
The Macdonald symmetric function $P_{\lambda/\mu}$ for a skew-partition $\lambda/\mu$ is defined by
$P_{\lambda}(X,Y)=\sum_{\mu\in\mbb{Y}}P_{\lambda/\mu}(X)P_{\mu}(Y)$.

\subsection{Definition using an inner product}
The Macdonald symmetric functions are also characterized as an orthogonal basis of $\Lambda$
with respect to an inner product defined below.
We write the inner product as $\braket{\cdot,\cdot}_{q,t}:\Lambda\times\Lambda\to \mbb{F}$ and define it as \cite[Section VI. 2]{Macdonald1995}
\begin{equation*}
	\braket{p_{\lambda},p_{\mu}}_{q,t}=z_{\lambda}(q,t)\delta_{\lambda,\mu},\ \ \lambda,\mu\in\mbb{Y},
\end{equation*}
where we set
\begin{equation*}
	z_{\lambda}(q,t)=z_{\lambda}\prod_{i=1}^{\ell(\lambda)}\frac{1-q^{\lambda_{i}}}{1-t^{\lambda_{i}}},\ \ z_{\lambda}=\prod_{i=1}^{\infty}m_{i}(\lambda)!i^{m_{i}(\lambda)},\ \ \lambda\in\mbb{Y}.
\end{equation*}
Then, the Macdonald symmetric functions $P_{\lambda}$, $\lambda\in\mbb{Y}$ are characterized by the triangularity (\ref{eq:triangularity}) and orthogonality: $\braket{P_{\lambda},P_{\mu}}=0$ if $\lambda\neq \mu$.

Note that, when we set $Q_{\lambda}:=\frac{1}{\braket{P_{\lambda},P_{\lambda}}_{q,t}}P_{\lambda}$, $\lambda\in\mbb{Y}$, the collection $\{Q_{\lambda}\}_{\lambda\in\mbb{Y}}$ is the dual basis of $\{P_{\lambda}\}_{\lambda\in\mbb{Y}}$ with respect to the inner product $\braket{\cdot,\cdot}_{q,t}$. 
Similarly to the usual Macdonald symmetric functions, we define $Q_{\lambda/\mu}$ for a skew-partition by $Q_{\lambda}(X,Y)=\sum_{\mu\in\mbb{Y}}Q_{\lambda/\mu}(X)Q_{\mu}(Y)$.

\subsection{Another series of operators}
\label{subsec:other_operators}
We also consider the following operators:
Let $r,n\in\mbb{Z}_{\ge 1}$ and define an operator on $\Lambda^{(n)}$ by
\begin{equation*}
	H_{r}^{(n)}:=\sum_{\substack{\nu\in (\mbb{Z}_{\ge 0})^{n}\\ |\nu|=r}}\left(\prod_{1\le i<j\le n}\frac{q^{\nu_{i}}x_{i}-q^{\nu_{j}}x_{j}}{x_{i}-x_{j}}\right)
		\left(\prod_{i,j=1}^{n}\frac{(tx_{i}/x_{j};q)_{\nu_{i}}}{(qx_{i}/x_{j};q)_{\nu_{i}}}\right)\prod_{i=1}^{n}T_{q,x_{i}}^{\nu_{i}}.
\end{equation*}
Here we wrote $\nu=(\nu_{1},\dots,\nu_{n})$ and $|\nu|=\sum_{i=1}^{n}\nu_{i}$.
It was announced in \cite[Proposition 3.24]{FeiginHashizumeHoshinoShiraishiYanagida2009} without a proof that these operators are also diagonalized by the Macdonald polynomials so that
\begin{equation*}
	H_{r}^{(n)}P^{(n)}_{\lambda}(x_{1},\dots, x_{n})=g^{(n)}_{r}(q^{\lambda_{1}}t^{n-1},\dots, q^{\lambda_{n}};q,t)P^{(n)}_{\lambda}(x_{1},\dots, x_{n})
\end{equation*}
for all $\lambda$ and $n\ge \ell(\lambda)$, where $g_{r}(X;q,t):=Q_{(r)}(X;q,t)$, $r=1,2,\dots$.
Later proofs appeared in \cite[Proposition 2.17]{BorodinCorwin2014} and  \cite[Section 5]{BorodinCorwinGorinShakirov2016}.
To enhance these operators to ones on $\Lambda$, we again have to consider their renormalized version. 
\begin{thm}[{\cite[Proposition 3.25]{FeiginHashizumeHoshinoShiraishiYanagida2009}}]
\label{thm:operator_G}
For  a fixed $r\in\mbb{Z}_{\ge 1}$ and $n\in\mbb{Z}_{\ge 1}$, we set
\begin{equation*}
	G^{(n)}_{r}:=\frac{(-1)^{r}t^{-nr}q^{\binom{r}{2}}}{(q;q)_{r}}\sum_{l=0}^{r}(-1)^{l}q^{-\binom{l}{2}}q^{-l(r-l)}(q^{-l+r-1};q)_{l}H^{(n)}_{l}.
\end{equation*}
Then the projective limit $G_{r}=G_{r}(q,t)=\varprojlim_{n}G^{(n)}_{r}$ exists and is diagonalized by the Macdonald symmetric functions so that
\begin{equation*}
	G_{r}P_{\lambda}(X)=g_{r}(q^{\lambda}t^{-\delta};q,t)P_{\lambda}(X),\ \ \lambda\in \mbb{Y}.
\end{equation*}
\end{thm}

\section{Free field realization}
\label{sect:free_field_realization}
In this section, we interpret the whole thing in Section \ref{sect:symm_funcs} in terms of the free field theory in two steps. In the former Subsection \ref{subsect:Fock_rep}, we identify the space of symmetric functions $\Lambda$ with a Fock representation in a standard manner \cite{Jing1994, AwataMatsuoOdakeShiraishi1995, AwataOdakeShiraishi1996, FeiginHashizumeHoshinoShiraishiYanagida2009,AwataFeiginHoshinoKanaiShiraishiYanagida2011}. The goal there is a proof of Theorem \ref{thm:correlation_operator}. The latter Subsection \ref{subsect:free_field_operators} is devoted to the realization of operators that are diagonalized by the Macdonald symmetric functions as operators on the Fock representation.

\subsection{Fock representation}
\label{subsect:Fock_rep}
Let $\mfrak{h}=\left(\bigoplus_{n\in\mbb{Z}\backslash\{0\}}\mbb{F}a_{n}\right)\oplus\mbb{F}c$ be a Heisenberg Lie algebra defined by
\begin{align}
\label{eq:deformed_Heisenberg}
	[a_{m},a_{n}]&=m\frac{1-q^{|m|}}{1-t^{|m|}}\delta_{m+n,0}c,\ \ m,n\in\mbb{Z}\backslash\{0\},& [c,\mfrak{h}]&=0.
\end{align}
We decompose the Heisenberg Lie algebra so that $\mfrak{h}=\mfrak{h}_{+}\oplus\mbb{F}c\oplus\mfrak{h}_{-}$, where $\mfrak{h}_{\pm}$ are Lie subalgebras generated by $\set{a_{\pm n}|n>0}$.
A one-dimensional representation $\mbb{F}\ket{0}$ of $\mfrak{h}_{\ge 0}:=\mfrak{h}_{+}\oplus\mbb{F}c$ is defined by the property that $a_{n}\ket{0}=0$, $n>0$ and $c\ket{0}=\ket{0}$.
The induced representation is the Fock representation of $\mfrak{h}$:
\begin{equation*}
	\mcal{F}:=U(\mfrak{h})\otimes_{U(\mfrak{h}_{\ge 0})}\mbb{F}\ket{0}\simeq U(\mfrak{h}_{-})\otimes_{\mbb{F}}\mbb{F}\ket{0}.
\end{equation*}
Here, for a Lie algebra $\mfrak{g}$, $U(\mfrak{g})$ is its universal enveloping algebra.
For a partition $\lambda=(\lambda_{1},\lambda_{2},\dots)\in\mbb{Y}$, we set $\ket{\lambda}:=a_{-\lambda_{1}}a_{-\lambda_{2}}\cdots\ket{0}$.
Then the Fock space has a basis $\{\ket{\lambda}:\lambda\in\mbb{Y}\}$.

The dual Fock space $\mcal{F}^{\dagger}$ is also constructed by induction.
Let $\mbb{F}\bra{0}$ be a one-dimensional right representation of $\mfrak{h}_{\le 0}=\mfrak{h}_{-}\oplus\mbb{F}c$ defined by $\bra{0}a_{-n}=0$, $n>0$ and $\bra{0}c=\bra{0}$.
Then the dual Fock space is obtained by
\begin{equation*}
	\mcal{F}^{\dagger}=\mbb{F}\bra{0}\otimes_{U(\mfrak{h}_{\le 0})}U(\mfrak{h})\simeq \mbb{F}\bra{0}\otimes_{\mbb{F}}U(\mfrak{h}_{+}).
\end{equation*}
For a partition $\lambda=(\lambda_{1},\lambda_{2},\dots )\in\mbb{Y}$, we set $\bra{\lambda}=\bra{0}a_{\lambda_{1}}a_{\lambda_{2}}\cdots$.
Then the collection $\{\bra{\lambda}:\lambda\in\mbb{Y}\}$ forms a basis of $\mcal{F}^{\dagger}$.

We define an $\mbb{F}$-bilinear paring $\braket{\cdot|\cdot}:\mcal{F}^{\dagger}\times \mcal{F}\to \mbb{F}$ by the properties $\braket{0|0}=1$ and
\begin{equation*}
	\braket{v|a_{n}\cdot|w}=\braket{v|\cdot a_{n}|w},\ \ \bra{v}\in\mcal{F}^{\dagger},\ \ \ket{w}\in\mcal{F},\ \ n\in\mbb{Z}\backslash\{0\}.
\end{equation*}
Then we have the following Propositions \ref{prop:symm_func_Fock} and \ref{prop:iota_vertex_expression}, which are well-known facts. See e.g. \cite[Section 3]{AwataOdakeShiraishi1996} for proofs.
\begin{prop}
\label{prop:symm_func_Fock}
The Fock space and the dual Fock space are isomorphic to the space of symmetric functions $\Lambda$ by the assignments
\begin{align*}
	\iota&:\mcal{F}\to\Lambda;\ \ \ket{\lambda}\mapsto p_{\lambda},&
	\iota^{\dagger}&:\mcal{F}^{\dagger}\to \Lambda;\ \ \bra{\lambda}\mapsto p_{\lambda}.
\end{align*}
Moreover, these assignments are compatible with the inner products so that the following diagram is commutative:
\begin{equation*}
\xymatrix{
	\mcal{F}^{\dagger}\otimes \mcal{F} \ar[rr]^{\iota^{\dagger}\otimes\iota} \ar[rd]_{\braket{\cdot|\cdot}}&& \Lambda \otimes \Lambda \ar[ld]^{\braket{\cdot,\cdot}_{q,t}}\\
	& \mbb{F} &.
}	
\end{equation*}
\end{prop}

\begin{prop}
\label{prop:iota_vertex_expression}
The mappings $\iota$ and $\iota^{\dagger}$ are equivalent to computation of the following matrix elements:
\begin{align*}
	\iota (\ket{v})&=\braket{0|\Gamma (X)_{+}|v},\ \ \ket{v}\in\mcal{F},&
	\iota^{\dagger} (\bra{v})&=\braket{v|\Gamma (X)_{-}|0},\ \ \bra{v}\in\mcal{F}^{\dagger},
\end{align*}
where $\Gamma (X)_{\pm}$ are defined in (\ref{eq:Gamma}).
\end{prop}

\begin{rem}
It can also be shown that
\begin{align*}
	\iota(a_{n}\ket{v})&=n\frac{1-q^{n}}{1-t^{n}}\frac{\del}{\del p_{n}}\iota(\ket{v}),\ \ \ket{v}\in\mcal{F},\ \ n>0,\\
	\iota^{\dagger}(\bra{v}a_{-n})&= n\frac{1-q^{n}}{1-t^{n}}\frac{\del}{\del p_{n}}\iota^{\dagger}(\bra{v}),\ \ \bra{v}\in\mcal{F}^{\dagger},\ \ n>0.
\end{align*}
\end{rem}

For a Macdonald symmetric function $P_{\lambda}\in\Lambda$, $\lambda\in\mbb{Y}$ and its dual $Q_{\lambda}\in\Lambda$,
we set
\begin{align*}
	\ket{P_{\lambda}}&:=\iota^{-1}(P_{\lambda})\in \mcal{F}, & \ket{Q_{\lambda}}&:=\iota^{-1}(Q_{\lambda})\in\mcal{F},\\
	\bra{P_{\lambda}}&:=(\iota^{\dagger})^{-1}(P_{\lambda})\in \mcal{F}^{\dagger}, & \bra{Q_{\lambda}}&:=(\iota^{\dagger})^{-1}(Q_{\lambda})\in\mcal{F}^{\dagger}.
\end{align*}
Then it follows from Proposition \ref{prop:symm_func_Fock} that $\braket{Q_{\lambda}|P_{\mu}}=\braket{P_{\lambda}|Q_{\mu}}=\delta_{\lambda,\mu}$.
These properties and Proposition \ref{prop:iota_vertex_expression} verify the following.

\begin{prop}
We have
\begin{align*}
	\bra{0}\Gamma(X)_{+}&=\sum_{\lambda\in\mbb{Y}}P_{\lambda}(X)\bra{Q_{\lambda}}, &
	\Gamma(X)_{-}\ket{0}&=\sum_{\lambda\in\mbb{Y}}P_{\lambda}(X)\ket{Q_{\lambda}}.
\end{align*}
\end{prop}

\begin{rem}
We can see that the computation of $\braket{0|\Gamma(Y)_{+}\Gamma(X)_{-}|0}$ reproduces the Cauchy-type identity.
Indeed, on one hand, it reads
\begin{equation*}
	\braket{0|\Gamma(Y)_{+}\Gamma(X)_{-}|0}=\sum_{\lambda\in\mbb{Y}}P_{\lambda}(X;q,t)Q_{\lambda}(Y;q,t).
\end{equation*}
On the other hand, a standard computation relying on the Baker--Campbell--Hausdorff formula gives
\begin{align*}
	\Gamma(Y)_{+}\Gamma(X)_{-}
	&=\exp\left(\sum_{n>0}\frac{1-t^{n}}{1-q^{n}}\frac{p_{n}(X)p_{n}(Y)}{n}\right)\Gamma(X)_{-}\Gamma(Y)_{+}\\
	&=\Pi(X,Y;q,t)\Gamma(X)_{-}\Gamma(Y)_{+}. \notag
\end{align*}
This reproduces the Cauchy-type identity
\begin{equation*}
	\sum_{\lambda\in\mbb{Y}}P_{\lambda}(X;q,t)Q_{\lambda}(Y;q,t)=\braket{0|\Gamma(Y)_{+}\Gamma(X)_{-}|0}=\Pi(X,Y;q,t).
\end{equation*}
\end{rem}

The Macdonald symmetric functions for skew-partitions are also expressed as matrix elements.
\begin{prop}[{\cite[Theorem 3.7]{AwataOdakeShiraishi1996}}]
\label{prop:skew-Macdonald_matrix_elem}
Let $\lambda/\mu$ be a skew-partition.
Then the corresponding Macdonald symmetric function has the following expressions.
\begin{equation*}
	P_{\lambda/\mu}(X)=\braket{Q_{\mu}|\Gamma(X)_{+}|P_{\lambda}}=\braket{P_{\lambda}|\Gamma (X)_{-}|Q_{\mu}}.
\end{equation*}
Also, the dual Macdonald symmetric function is expressed as
\begin{equation*}
	Q_{\lambda/\mu}(X)=\braket{P_{\mu}|\Gamma (X)_{+}|Q_{\lambda}}=\braket{Q_{\lambda}|\Gamma (X)_{-}|P_{\mu}}.
\end{equation*}
\end{prop}
Precisely speaking, \cite[Theorem 3.7]{AwataOdakeShiraishi1996} is a statement only about $Q_{\lambda/\mu}$, but the assertion for $P_{\lambda/\mu}$ also holds by the same reasoning.

We are in position to prove Theorem \ref{thm:correlation_operator}.
\begin{proof}[Proof of Theorem \ref{thm:correlation_operator}]
We first compute the following matrix element:
\begin{align*}
	\braket{P_{\mu}|\psi_{f}^{X,Y}|Q_{\nu}}
	&=\braket{P_{\mu}|\Gamma (Y)_{+}\mcal{O}(f)\Gamma (X)_{-}|Q_{\nu}}=\sum_{\lambda\in\mbb{Y}}f(\lambda)P_{\lambda/\nu}(X)Q_{\lambda/\mu}(Y).	
\end{align*}
Here we inserted the identity operator $\mrm{Id}_{\mcal{F}}=\sum_{\lambda\in\mbb{Y}}\ket{Q_{\lambda}}\bra{P_{\lambda}}$ into the left and right of the operator $\mcal{O}(f)$ and used Proposition \ref{prop:skew-Macdonald_matrix_elem}.
Insertion of the identity operator also gives us
\begin{align*}
	&\braket{0|\psi_{f_{N}}^{X^{(N)},Y^{(N)}}\cdots \psi_{f_{1}}^{X^{(1)},Y^{(1)}}|0} \\
	&=\sum_{\nu^{(1)},\dots, \nu^{(N-1)}\in\mbb{Y}}\braket{0|\psi_{f_{N}}^{X^{(N)},Y^{(N)}}|Q_{\nu^{(N-1)}}} \braket{P_{\nu^{(N-1)}}|\psi_{f_{N-1}}^{X^{(N-1)},Y^{(N-1)}}|Q_{\nu^{(N-2)}}}\\
	&\hspace{80pt}\times \cdots \times \braket{P_{\nu^{(1)}}|\psi_{f_{1}}^{X^{(1)},Y^{(1)}}|0} \\
	&= \sum_{\nu^{(1)},\dots, \nu^{(N-1)}\in\mbb{Y}}\sum_{\lambda^{(1)},\dots,\lambda^{(N)}\in\mbb{Y}}
		f_{N}(\lambda^{(N)})P_{\lambda^{(N)}/\nu^{(N-1)}}(X^{(N)})Q_{\lambda^{(N)}/\emptyset}(Y^{(N)})\\
	&\hspace{100pt}\times f_{N-1}(\lambda^{(N-1)})P_{\lambda^{(N-1)}/\nu^{(N-2)}}(X^{(N-1)})Q_{\lambda^{(N-1)}/\nu^{(N-1)}}(Y^{(N-1)})\\
	&\hspace{120pt}\cdots\cdots \\
	&\hspace{100pt}\times f_{1}(\lambda^{(1)})P_{\lambda^{(1)}/\emptyset} (X^{(1)})Q_{\lambda^{(1)}/\nu^{(1)}}(Y^{(1)}) \\
	&=\sum_{\lambda^{(1)},\dots,\lambda^{(N)}\in\mbb{Y}}f_{1}(\lambda^{(1)})\cdots f_{N}(\lambda^{(N)})P_{\lambda^{(1)}}(X^{(1)})\Psi_{\lambda^{(1)},\lambda^{(2)}}(Y^{(1)},X^{(2)}) \\
	&\hspace{100pt}\times \cdots \times \Psi_{\lambda^{(N-1)},\lambda^{(N)}}(Y^{(N-1)},X^{(N)})Q_{\lambda^{(N)}}(Y^{(N)}).
\end{align*}
Here we used the definition of the transition function (\ref{eq:transition_function}).
In the case when $f_{1},\dots, f_{N}=1$,
the above result exactly gives the normalization factor:
\begin{equation*}
	\braket{0|\psi_{1}^{X^{(N)},Y^{(N)}}\cdots \psi_{1}^{X^{(1)},Y^{(1)}}|0}=\prod_{1\le i\le j\le N}\Pi (X^{(i)},Y^{(j)}).
\end{equation*}
Therefore, we have
\begin{align*}
	&\frac{\braket{0|\psi_{f_{N}}^{X^{(N)},Y^{(N)}}\cdots \psi_{f_{1}}^{X^{(1)},Y^{(1)}}|0}} {\braket{0|\psi_{1}^{X^{(N)},Y^{(N)}}\cdots \psi_{1}^{X^{(1)},Y^{(1)}}|0}}=\sum_{\lambda^{(1)},\dots,\lambda^{(N)}\in\mbb{Y}}f_{1}(\lambda^{(1)})\cdots f_{N}(\lambda^{(N)})\mbb{MP}^{N}_{q,t}(\lambda^{(1)},\dots,\lambda^{(N)}),
\end{align*}
which is just the correlation function $\mbb{E}_{q,t}^{N}[f_{1}[1]\cdots f_{N}[N]]$.
\end{proof}

\subsection{Free field realization of operators}
\label{subsect:free_field_operators}
First, let us fix our terminologies.
\begin{defn}
Let $T\in\mrm{End}(\Lambda)$ be an operator on $\Lambda$.
We call the operator
\begin{equation*}
	\what{T}:=\iota^{-1} T \iota\in\mrm{End}(\mcal{F})
\end{equation*}
the free field realization of $T$.
\end{defn}

\begin{defn}
\label{defn:vertex_operator}
Let $z$ be a formal variable.
An $\mrm{End}(\mcal{F})$-valued formal power series $V(z)$ of the following form is called a vertex operator:
\begin{align*}
	V(z)&=V(z)_{-}V(z)_{+}\in\mrm{End}(\mcal{F})[[z,z^{-1}]],\\
	V(z)_{+}&=\exp\left(\sum_{n>0}\gamma_{n}a_{n}z^{-n}\right),\ \ \gamma_{n}\in\mbb{F},\ \ n>0,\\
	V(z)_{-}&=\exp\left(\sum_{n>0}\gamma_{-n}a_{-n}z^{n}\right),\ \ \gamma_{-n}\in\mbb{F},\ \ n>0
\end{align*}
\end{defn}
\begin{rem}
A vertex operator does not converge in $U(\mfrak{h})[[z,z^{-1}]]$,
but makes sense in $\mrm{End}(\mcal{F})[[z,z^{-1}]]$.
\end{rem}

\begin{defn}
\label{defn:normal_order_vertex_operator}
Let $V_{i}(z_{i})\in\mrm{End}(\mcal{F})[[z_{i},z_{i}^{-1}]]$, $i=1,2,\dots, r$ be vertex operators.
Their normally-ordered product $\no{V_{1}(z_{1})\cdots V_{r}(z_{r})}\in\mrm{End}(\mcal{F})[[z_{i},z_{i}^{-1}|i=1,\dots,r]]$ is defined as
\begin{equation*}
	\no{V_{1}(z_{1})\cdots V_{r}(z_{r})}
	:=V_{1}(z_{1})_{-}\cdots V_{r}(z_{r})_{-}V_{1}(z_{1})_{+}\cdots V_{r}(z_{r})_{+}.
\end{equation*}
\end{defn}

In this subsection, we see the free field realization of operators introduced in Section \ref{sect:symm_funcs}
that are diagonalized by the Macdonald symmetric functions.
We begin with the Macdonald operators $E_{r}$, $r=1,2,,\dots$.
Notice that the vertex operator $\eta(z)$ in (\ref{eq:vertex_eta}) gives an example of Definition \ref{defn:vertex_operator}, hence, in particular, the normally ordered product is defined for it.

\begin{thm}[{\cite[Theorem 9.2]{Shiraishi2006}, \cite[Proposition 3.6]{FeiginHashizumeHoshinoShiraishiYanagida2009}}]
\label{thm:Macdonald_operators_free_field}
Let $r=1,2,\dots$.
The following operator gives the free field realization of $E_{r}$:
\begin{equation*}
	\what{E}_{r}=\frac{t^{-r(r+1)/2}}{(t^{-1};t^{-1})_{r}}\int \left(\prod_{i=1}^{r}\frac{dz_{i}}{2\pi\sqrt{-1}z_{i}}\right)\left(\prod_{1\le i<j\le r}\frac{1-z_{j}/z_{i}}{1-t^{-1}z_{j}/z_{i}}\right)\no{\eta(z_{1})\cdots \eta(z_{r})}.
\end{equation*}
In the case of $r=1$, the product part $\prod_{1\le i<j\le r}\frac{1-z_{j}/z_{i}}{1-t^{-1}z_{j}/z_{i}}$ is understood as unity.
Here the linear functional $\int\left(\prod_{i=1}^{r}\frac{dz_{i}}{2\pi\sqrt{-1}z_{i}}\right)$ takes the constant term.
\end{thm}

On the same space $\mcal{F}$, the Macdonald operators $E_{r}(q^{-1},t^{-1})$, $r=1,2,\dots$ with inverted parameters are also realized.
To this end, we introduce another vertex operator:
\begin{equation*}
	\xi(z)=\exp\left(-\sum_{n>0}\frac{1-t^{-n}}{n}(t/q)^{n/2}a_{-n}z^{n}\right)\exp\left(\sum_{n>0}\frac{1-t^{n}}{n}(t/q)^{n/2}a_{n}z^{-n}\right).
\end{equation*}
Notice that, though we see the contribution from $(q/t)^{1/2}$, each coefficient of $z^{n}$, $n\in\mbb{Z}$ makes sense over the base field $\mbb{F}$.

\begin{thm}[{\cite[Section III. F]{FeiginHashizumeHoshinoShiraishiYanagida2009}}]
For $r=1,2,\dots$, the operator
\begin{equation*}
	\what{E}_{r}(q^{-1},t^{-1})=\frac{t^{r(r+1)/2}}{(t;t)_{r}}\int\left(\prod_{i=1}^{r}\frac{dz_{i}}{2\pi\sqrt{-1}z_{i}}\right)\left(\prod_{1\le i<j\le r}\frac{1-z_{j}/z_{i}}{1-tz_{j}/z_{i}}\right)\no{\xi(z_{1})\cdots \xi(z_{r})}
\end{equation*}
gives the free field realization of $E_{r}(q^{-1},t^{-1})$.
\end{thm}

We have also introduced the operators $G_{r}$, $r=1,2,\dots$ that are also diagonalized by the Macdonald symmetric functions (Subsection \ref{subsec:other_operators}).
These operators admit the following free field realization:
\begin{thm}[{\cite[Proposition 3.17]{FeiginHashizumeHoshinoShiraishiYanagida2009}}]
For each $r=1,2,\dots$, the following operators on $\mcal{F}$ is the free field realization of $G_{r}$:
\begin{equation*}
	\what{G}_{r}=\frac{(-1)^{r}q^{\binom{r}{2}}}{(q;q)_{r}}\int\left(\prod_{i=1}^{r}\frac{dz_{i}}{2\pi\sqrt{-1}z_{i}}\right)\prod_{1\le i<j\le r}\frac{1-z_{j}/z_{i}}{1-qz_{j}/z_{i}}\no{\eta(z_{1})\cdots\eta(z_{r})}.
\end{equation*}
\end{thm}

The inversion of the parameters $q\to q^{-1}$, $t\to t^{-1}$ again involves the replacement of vertex operators $\eta(z) \to \xi(z)$.
\begin{thm}
\label{thm:free_field_G_inverse}
For each $r=1,2,\dots,$ the following operator on $\mcal{F}$ is the free field realization of $G_{r}(q^{-1},t^{-1})$:
\begin{equation*}
	\what{G}_{r}(q^{-1},t^{-1})=\frac{(-1)^{r}q^{-\binom{r}{2}}}{(q^{-1};q^{-1})_{r}}\int\left(\prod_{i=1}^{r}\frac{dz_{i}}{2\pi\sqrt{-1}z_{i}}\right)\prod_{1\le i<j\le r}\frac{1-z_{j}/z_{i}}{1-q^{-1}z_{j}/z_{i}}\no{\xi(z_{1})\cdots\xi(z_{r})}.
\end{equation*}
\end{thm}
Though naturally expected from \cite{FeiginHashizumeHoshinoShiraishiYanagida2009}, this result is not stated nor proved therein. We will give a proof of Theorem \ref{thm:free_field_G_inverse} in Appendix \ref{app:proof_free_field_G_inverse}.

\section{Determinantal expression and the Schur-limit}
\label{sect:determinant}
In this section, we derive alternative expressions of the free field realizations presented in the previous Section \ref{sect:free_field_realization} by using determinants to prove Theorem \ref{thm:Macdonald_determinant} and the determinantal expressions of other operators. We also discuss the Schur-limit to see that the determinant gains a natural interpretation in terms of free fermions.

\subsection{Preliminaries}
We introduce the operation of symmetrization.
\begin{defn}
For a function $f(z_{1},\dots, z_{n})$ of $n$ variables $z_{1},\dots, z_{n}$,
the symmetrization is defined by
\begin{equation*}
	\mrm{Sym}[f(z_{1},\dots, z_{n})]:=\frac{1}{n!}\sum_{\sigma\in \mfrak{S}_{n}}f(z_{\sigma (1)},\dots, z_{\sigma (n)}).
\end{equation*}
\end{defn}

The following lemma has been observed e.g. in the proof of \cite[Lemma 9.4]{Shiraishi2006} and plays a key role in this section.
\begin{lem}
\label{lem:key_symmetrization}
We have
\begin{equation*}
	\mrm{Sym}\left(\prod_{1\le i<j\le n}\frac{z_{i}-z_{j}}{z_{i}-tz_{j}}\right)=\frac{[n]_{t}!}{n!}\prod_{1\le i<j\le n}\frac{(z_{i}-z_{j})(z_{j}-z_{i})}{(z_{i}-tz_{j})(z_{j}-tz_{i})}.
\end{equation*}
\end{lem}
\begin{proof}
The Hall--Littlewood polynomial for the empty partition reads \cite[Chapter III, (1.4)]{Macdonald1995}
\begin{equation*}
	\sum_{\sigma\in\mfrak{S}_{n}}\sigma\left(\prod_{i<j}\frac{z_{i}-tz_{j}}{z_{i}-z_{j}}\right)=[n]_{t}!.
\end{equation*}
When we multiply $\prod_{i<j}\frac{(z_{i}-z_{j})(z_{j}-z_{i})}{(z_{i}-tz_{j})(z_{j}-tz_{i})}$ on the both sides, we obtain
\begin{equation*}
	\sum_{\sigma\in\mfrak{S}_{n}}\prod_{i<j}\frac{z_{\sigma(j)}-z_{\sigma(i)}}{z_{\sigma(j)}-tz_{\sigma(i)}}=[n]_{t}!\prod_{i<j}\frac{(z_{i}-z_{j})(z_{j}-z_{i})}{(z_{i}-tz_{j})(z_{j}-tz_{i})}.
\end{equation*}
When we write the longest element in $\mfrak{S}_{n}$ as,
\begin{equation*}
	\sigma^{\ast}=\left(
	\begin{array}{cccc}
	1 & 2 & \cdots & n \\
	n & n-1 & \cdots &1
	\end{array}
	\right),
\end{equation*}we see that
\begin{equation*}
	\sum_{\sigma\in\mfrak{S}_{n}}\sigma\circ\sigma^{\ast}\left(\prod_{i<j}\frac{z_{i}-z_{j}}{z_{i}-tz_{j}}\right)=[n]_{t}!\prod_{i<j}\frac{(z_{i}-z_{j})(z_{j}-z_{i})}{(z_{i}-tz_{j})(z_{j}-tz_{i})},
\end{equation*}
where the left hand side is $n!\mrm{Sym}\left(\prod_{i<j}\frac{z_{i}-z_{j}}{z_{i}-tz_{j}}\right)$.
\end{proof}

\subsection{Proof of Theorem \ref{thm:Macdonald_determinant} and other operators}
Now we are in position to prove Theorem \ref{thm:Macdonald_determinant}.
\begin{proof}[Proof of Theorem \ref{thm:Macdonald_determinant}]
Recall that the constant term of a multivariable function is invariant under symmetrization.
Therefore we have
\begin{align*}
	\what{E}_{r}=\frac{t^{-r(r+1)/2}}{(t^{-1};t^{-1})_{r}}\int\left(\prod_{i=1}^{r}\frac{dz_{i}}{2\pi\sqrt{-1}z_{i}}\right)
				\mrm{Sym}\left(\prod_{1\le i<j\le r}\frac{1-z_{j}/z_{i}}{1-t^{-1}z_{j}/z_{i}}\right)\no{\eta(z_{1})\cdots \eta(z_{r})}.
\end{align*}
Note that the normally-ordered product of vertex operators is symmetric under exchange of variables.
As we saw in Lemma \ref{lem:key_symmetrization},
\begin{equation*}
	\mrm{Sym}\left(\prod_{1\le i<j\le r}\frac{1-z_{j}/z_{i}}{1-t^{-1}z_{j}/z_{i}}\right)
	=\frac{[r]_{t^{-1}}!}{r!}\prod_{1\le i<j\le r}\frac{(1-z_{j}/z_{i})(1-z_{i}/z_{j})}{(1-t^{-1}z_{j}/z_{i})(1-t^{-1}z_{i}/z_{j})}.
\end{equation*}
Thus
\begin{align*}
	\what{E}_{r}
	=\frac{t^{-r(r+1)/2}}{(t^{-1};t^{-1})_{r}}\frac{[r]_{t^{-1}}!}{r!}\int\left(\prod_{i=1}^{r}\frac{dz_{i}}{2\pi\sqrt{-1}z_{i}}\right)
				&\prod_{1\le i<j\le r}\frac{(1-z_{j}/z_{i})(1-z_{i}/z_{j})}{(1-t^{-1}z_{j}/z_{i})(1-t^{-1}z_{i}/z_{j})}\\
	&\times\no{\eta(z_{1})\cdots \eta(z_{r})}.
\end{align*}
The product can be further computed as
\begin{align*}
	&\prod_{1\le i<j\le r}\frac{(1-z_{j}/z_{i})(1-z_{i}/z_{j})}{(1-t^{-1}z_{j}/z_{i})(1-t^{-1}z_{i}/z_{j})}\\
	&=\frac{\prod_{i<j}(z_{i}-z_{j})\prod_{i<j}(z_{j}-z_{i})}{\prod_{i<j}(z_{i}-t^{-1}z_{j})\prod_{i<j}(z_{j}-t^{-1}z_{i})} \\
	&=t^{r(r-1)/2}(1-t^{-1})^{r}\prod_{i=1}^{r}z_{i}\frac{\prod_{i<j}(z_{i}-z_{j})\prod_{i<j}(t^{-1}z_{j}-t^{-1}z_{i})}{\prod_{i,j}(z_{i}-t^{-1}z_{j})}.
\end{align*}
Now recall the Cauchy determinant formula
\begin{equation*}
	\det\left(\frac{1}{x_{i}-y_{j}}\right)_{i,j}=\frac{\prod_{i<j}(x_{i}-x_{j})\prod_{i<j}(y_{j}-y_{i})}{\prod_{i,j}(x_{i}-y_{j})}.
\end{equation*}
Using this, we obtain
\begin{align*}
	\what{E}_{r}
	&=\frac{t^{-r(r+1)/2}}{(t^{-1};t^{-1})_{r}}\frac{[r]_{t^{-1}}!}{r!}\int\left(\prod_{i=1}^{r}\frac{dz_{i}}{2\pi\sqrt{-1}z_{i}}\right)
		t^{r(r-1)/2}(1-t^{-1})^{r}\prod_{i=1}^{r}z_{i}\\
	&\hspace{120pt}\times\det\left(\frac{1}{z_{i}-t^{-1}z_{j}}\right)_{1\le i,j\le r}\no{\eta(z_{1})\cdots\eta(z_{r})}\\
	&=\frac{t^{-r}}{r!}\int\left(\prod_{i=1}^{r}\frac{dz_{i}}{2\pi\sqrt{-1}}\right)\det\left(\frac{1}{z_{i}-t^{-1}z_{j}}\right)_{1\le i,j\le r}
				\no{\eta(z_{1})\cdots \eta(z_{r})}.
\end{align*}
This is the desired result.
\end{proof}

In similar manners, we can also derive the following expressions.
\begin{thm}
\label{thm:others_determinant}
For $r=1,2,\dots$, we have
\begin{align*}
	\what{E}_{r}(q^{-1},t^{-1})&=\frac{t^{r}}{r!}\int\left(\prod_{i=1}^{r}\frac{dz_{i}}{2\pi\sqrt{-1}}\right)\det\left(\frac{1}{z_{i}-tz_{j}}\right)_{1\le i,j\le r}\no{\xi(z_{1})\cdots \xi(z_{r})}, \\
	\what{G}_{r}&=\frac{(-1)^{r}}{r!}\int \left(\prod_{i=1}^{r}\frac{dz_{i}}{2\pi\sqrt{-1}}\right)\det\left(\frac{1}{z_{i}-qz_{j}}\right)_{1\le i,j\le r}\no{\eta(z_{1})\cdots\eta(z_{r})}, \\
	\what{G}_{r}(q^{-1},t^{-1})&=\frac{(-1)^{r}}{r!}\int \left(\prod_{i=1}^{r}\frac{dz_{i}}{2\pi\sqrt{-1}}\right)\det\left(\frac{1}{z_{i}-q^{-1}z_{j}}\right)_{1\le i,j\le r}\no{\xi(z_{1})\cdots\xi(z_{r})}
\end{align*}
\end{thm}

Let us name a frequently used functional; for $r=1,2,\dots$ and $\gamma\in\mbb{F}$, we define a functional
\begin{equation*}
	\int D^{r}_{\gamma}z:=\int \left(\prod_{i=1}^{r}\frac{dz_{i}}{2\pi\sqrt{-1}}\right)\det\left(\frac{1}{z_{i}-\gamma z_{j}}\right): \mbb{F}[[z_{i},z_{i}^{-1}|i=1,\dots, r]]\to\mbb{F}.
\end{equation*}
The following property will be used in the next Section \ref{sect:expectation}.
\begin{lem}
\label{lem:measure_invariance}
The functional $\int D^{r}_{\gamma}z$ is invariant under a uniform scale transformation and the uniform inversion.
Namely, if we set
\begin{equation*}
	w_{i}=\alpha z_{i},\ \  i=1,\dots, r,\ \ \alpha\in\mbb{F}
\end{equation*}
or 
\begin{equation*}
	w_{i}=z_{i}^{-1},\ \ i=1,\dots, r,
\end{equation*}
then we have
\begin{equation*}
	\int D^{r}_{\gamma}z=\int D^{r}_{\gamma}w.
\end{equation*}
\end{lem}
\begin{proof}
The invariance under a scale transformation is obvious.
We consider the case of inversion: $w_{i}=z_{i}^{-1}$, $i=1,\dots, n$.
The operation taking the residues in $z_{i}$, $i=1,\dots, r$ is written as
\begin{equation*}
	\int \left(\prod_{i=1}^{r}\frac{dz_{i}}{2\pi\sqrt{-1}}\right)=\int \left(\prod_{i=1}^{r}\frac{dw_{i}}{2\pi\sqrt{-1}}\right)\prod_{i=1}^{r}w_{i}^{-2}.
\end{equation*}
The determinant gives
\begin{align*}
	\det\left(\frac{1}{z_{i}-\gamma z_{j}}\right)_{1\le i,j\le r}
	=\det\left(\frac{w_{i}w_{j}}{w_{j}-\gamma w_{i}}\right)_{1\le i,j\le r}
	=\prod_{i=1}^{r}w_{i}^{2}\det \left(\frac{1}{w_{i}-\gamma w_{j}}\right)_{1\le i,j\le r}.
\end{align*}
Then the desired invariance is proved.
\end{proof}

\subsection{Schur-limit}
Throughout this subsection, we fix $t$ and refer to the limit $q\to t$ as the Schur-limit.
To illustrate the origin of the determinantal expressions, we investigate the Schur-limit of the operators discussed above. To this aim, we introduce a semi-infinite wedge space and free fermion fields on it following \cite[Appendix A]{Okounkov2001}. Let $V$ be a complex vector space spanned by basis vectors $\underline{k}$, $k\in\mbb{Z}+\frac{1}{2}$. We say that a subset $S\subset\mbb{Z}+\frac{1}{2}$ is a Maja diagram if it possesses the following two properties: (1) $S_{+}:=S\backslash \left(\mbb{Z}_{\le 0}-\frac{1}{2}\right)$ is finite, (2) $S_{-}:=\left(\mbb{Z}_{\le 0}-\frac{1}{2}\right)\backslash S$ is finite, and write $\mfrak{M}$ for the collection of Maja diagrams. Writing elements in a Maja diagram $S\in\mfrak{M}$ in the descending order as $S=\set{s_{1}>s_{2}>\cdots}$, we associate to it a semi-infinite wedge $v_{S}:=\underline{s_{1}}\wedge\underline{s_{2}}\wedge\cdots$. The semi-infinite wedge of $V$ is defined by $\Xi=\bigwedge^{\infty/2}V:=\bigoplus_{S\in\mfrak{M}}\mbb{C}v_{S}$ that is equipped with a bilinear form such that $\braket{v_{S}|v_{\pr{S}}}=\delta_{S,\pr{S}}$. For each $r\in\mbb{Z}+\frac{1}{2}$, we define an operator $\psi_{r}$ on $\Xi$ by $\psi_{r}\cdot v:=\underline{r}\wedge v$ and $\psi_{r}^{\ast}$ as its adjoint operator. Then these operators exhibit the canonical anti-commutation relations:
\begin{equation*}
	\{\psi_{r},\psi_{s}\}=\{\psi^{\ast}_{r},\psi^{\ast}_{s}\}=0,\ \ \{\psi_{r},\psi^{\ast}_{s}\}=\delta_{r,s},\ \ r,s\in\mbb{Z}+\frac{1}{2},
\end{equation*}
where $\{A,B\}=AB+BA$ is the anti-commutator. We define the charge operator $C$ by $C\cdot v_{S}:=\left(|S_{+}|-|S_{-}|\right)v_{S}$, $S\in\mfrak{M}$, and the translation operator $T$ by $T\cdot \underline{s_{1}}\wedge\underline{s_{2}}\wedge\cdots:=\underline{s_{1}+1}\wedge\underline{s_{2}+1}\wedge\cdots$, $S=\set{s_{1}>s_{2}>\cdots}\in\mfrak{M}$. It is obvious that the translation operator is invertible. We further introduce operators $\alpha_{n}$, $n\in\mbb{Z}\backslash\set{0}$ by $\alpha_{n}=\sum_{r\in\mbb{Z}+\frac{1}{2}}\psi_{r-n}\psi^{\ast}_{r}$. Then, they satisfy commutation relations $[\alpha_{m},\alpha_{n}]=m\delta_{m+n,0}$, $m,n\in\mbb{Z}\backslash\set{0}$, which are identified with the Schur-limit $q\to t$ of the Heisenberg commutation relations in (\ref{eq:deformed_Heisenberg}) at $c=1$. Due to the boson-fermion correspondence (see e.g. \cite[Lecture 5]{KacRainaRozhkovskaya2013}), the fermion operators are recovered by means of these boson operators, the charge and the translation operators. In terms of generating series $\psi(z)=\sum_{r\in\mbb{Z}+\frac{1}{2}}\psi_{r}z^{r-\frac{1}{2}}$ and $\psi^{\ast}(z)=\sum_{r\in\mbb{Z}+\frac{1}{2}}\psi_{r}^{\ast}z^{-r-\frac{1}{2}}$, we have
\begin{align*}
	\psi (z)&=Tz^{C}\exp\left(\sum_{n>0}\frac{\alpha_{-n}}{n}z^{n}\right)\left(-\sum_{n>0}\frac{\alpha_{n}}{n}z^{-n}\right), \\
	\psi^{\ast} (z)&=T^{-1}z^{-C}\exp\left(-\sum_{n>0}\frac{\alpha_{-n}}{n}z^{n}\right)\left(\sum_{n>0}\frac{\alpha_{n}}{n}z^{-n}\right).
\end{align*}

To a partition $\lambda\in \mbb{Y}$, we associate a Maja diagram ${\bf M}(\lambda)=\Set{\lambda_{i}-i+\frac{1}{2}}_{i\ge 1}$. Then, we can see that $|{\bf M}(\lambda)_{+}|=|{\bf M}(\lambda)_{-}|$. In fact, the assignment $\lambda\mapsto {\bf M}(\lambda)$ is a bijection between $\mbb{Y}$ and $\mfrak{M}_{0}=\set{S\in\mfrak{M}| |S_{+}|=|S_{-}|}$. Furthermore, the charge-zero subspace $\Xi_{0}:=\Set{v\in\Xi|C\cdot v=0}$ is isomorphic to $\Lambda$, under which the vector $v_{{\bf M}(\lambda)}$ is identified with the Schur function $s_{\lambda}$ of the same partition.

For each $r=1,2,\dots$, we define operators on $\Xi$ by
\begin{align*}
	{\bf E}_{r}&:=\frac{t^{-r}}{r!}\int\left(\prod_{i=1}^{r}\frac{dz_{i}}{2\pi\sqrt{-1}}\right)\psi (z_{1})\cdots \psi (z_{r})\psi^{\ast}(t^{-1}z_{r})\cdots \psi^{\ast}(t^{-1}z_{1}), \\
	{\bf G}_{r}&:=\frac{(-t)^{-r}}{r!}\int\left(\prod_{i=1}^{r}\frac{dz_{i}}{2\pi\sqrt{-1}}\right)\psi^{\ast}(t^{-1}z_{1})\cdots \psi^{\ast}(t^{-1}z_{r})\psi (z_{r})\cdots \psi (z_{1}).
\end{align*}
Notice that these operators commute with the charge operator $C$, and hence, can be regarded as operators on $\Lambda$.
The operator ${\bf E}_{1}$ is essentially the same as ``the operator $\mcal{E}_{0}(z)$" in \cite{OkounkovPandharipande2006} and is the action of an element in (completion of) the $\mcal{W}_{1+\infty}$-algebra.

\begin{prop}
\label{prop:Schur-limit_Macdonald}
For every $r=1,2,\dots$, the operators $\what{E}_{r}$ and $\what{G}_{r}$ on $\Lambda$ reduce at the Schur-limit $q\to t$ to ${\bf E}_{r}$ and ${\bf G}_{r}$, respectively.
\end{prop}
\begin{proof}
We only show the case of $\what{E}_{r}$ since that of $\what{G}_{r}$ is shown in a parallel argument.
We write the Schur-limit of the vertex operator $\eta (z)$ as
\begin{equation*}
	\eta^{\mrm{S}}(z)=\exp\left(\sum_{n>0}\frac{1-t^{-n}}{n}\alpha_{-n}z^{n}\right)\exp\left(-\sum_{n>0}\frac{1-t^{n}}{n}\alpha_{n}z^{-n}\right).
\end{equation*}
Then the normally ordered product of $\eta^{\mrm{S}}(z)$ makes sense in exactly the same manner as in Definition \ref{defn:normal_order_vertex_operator}. Due to the Wick formula, we have
\begin{equation*}
	\psi (z_{1})\cdots \psi (z_{r})\psi^{\ast}(t^{-1}z_{r})\cdots \psi^{\ast}(t^{-1}z_{1})=\det\left(\frac{1}{z_{i}-t^{-1}z_{j}}\right)_{1\le i,j\le r}t^{rC}\no{\eta^{\mrm{S}}(z_{1})\cdots \eta^{\mrm{S}}(z_{r})}.
\end{equation*}
Notice that the part $t^{rC}$ acts as unity restricted on the charge-zero subspace $\Xi_{0}$. We can see the desired result by comparing this with the expression in Theorem \ref{thm:Macdonald_determinant}.
\end{proof}

This result suggests that we may understand the free field realization $\what{E}_{r}$ and $\what{G}_{r}$, $r=1,2,\dots$ as the result of (1) rearranging the product of fermionic fields into the normal order in the bosonic basis, (2) deforming $\eta^{\mrm{S}}(z)$ to $\eta(z)$, or equivalently, $\alpha_{m}$ to $a_{m}$, $m\in\mbb{Z}\backslash\set{0}$, and (3) taking the residues. The determinants appearing in the expression of $\what{E}_{r}$ and $\what{G}_{r}$, $r=1,2,\dots$ have their origins at the step (1) of this procedure. A similar interpretation has been given in \cite{Prochazka2019} to the Nazarov--Sklyanin operator in the Jack case \cite{NazarovSklyanin2013a}.

We can directly compute the eigenvalues of ${\bf E}_{r}$, $r=1,2,\dots$ on the Schur basis. In fact, evaluating the residues, we have
\begin{equation*}
	{\bf E}_{r}=\frac{t^{rC}}{r!}\sum_{s_{1},\dots, s_{r}\in\mbb{Z}+\frac{1}{2}}t^{\sum_{i=1}^{r}\left(s_{i}-\frac{1}{2}\right)}\psi_{s_{1}}\cdots \psi_{s_{r}}\psi^{\ast}_{s_{r}}\cdots \psi^{\ast}_{s_{1}}.
\end{equation*}
From the very definition of a vector $v_{{\bf M}(\lambda)}$ for a partition $\lambda\in\mbb{Y}$, it follows that
\begin{equation*}
	{\bf E}_{r}v_{{\bf M}(\lambda)}=\frac{1}{r!}\sum_{\substack{i_{1},\dots,i_{r}\\ \mrm{distinct}}}t^{\sum_{j=1}^{r}\left(\lambda_{i_{j}}-i_{j}\right)}v_{{\bf M}(\lambda)}=e_{r}(t^{\lambda-\delta})v_{{\bf M}(\lambda)},
\end{equation*}
where $t^{\lambda-\delta}$ is the specialization defined by $x_{i}\mapsto t^{\lambda_{i}-i}$, $i\ge 1$. Note that the eigenvalue is the Schur limit of $e_{r}(q^{\lambda}t^{-\delta})$ as was expected.
As for the operator ${\bf G}_{r}$, we have
\begin{equation*}
	{\bf G}_{r}=\frac{(-1)^{r}}{r!}\sum_{s_{1},\dots, s_{r}\in\mbb{Z}+\frac{1}{2}}t^{\sum_{i=1}^{r}\left(s_{i}-\frac{1}{2}\right)}\psi^{\ast}_{s_{1}}\cdots \psi^{\ast}_{s_{r}}\psi_{s_{r}}\cdots \psi_{s_{1}},
\end{equation*}
and hence, ${\bf G}_{r}v_{{\bf M}(\lambda)}=(-1)^{r}e_{r}(t^{-\pr{\lambda}+\delta-1})v_{{\bf M}(\lambda)}$, $\lambda\in\mbb{Y}$, where $\pr{\lambda}$ is the transpose of $\lambda$ and the specialization is $x_{i}\mapsto t^{-\pr{\lambda}+i-1}_{i}$, $i\ge 1$. The eigenvalue can be shown to coincide with the Schur-limit of $g_{r}(q^{\lambda}t^{-\delta};q,t)$ due to the following lemma that is a special case of \cite[Lemma 3.27]{FeiginHashizumeHoshinoShiraishiYanagida2009}:
\begin{lem}
For each $\lambda\in\mbb{Y}$ and $r=1,2,\dots$, the identity $h_{r}(t^{\lambda-\delta})=(-1)^{r}e_{r}(t^{-\pr{\lambda}+\delta-1})$ holds, where $h_{r}=\lim_{q\to t}g_{r}(q,t)$ is the $r$-th complete symmetric function.
\end{lem}
\begin{proof}
The generating series of the complete symmetric functions reads $\sum_{r=0}^{\infty}h_{r}(X)u^{r}=\prod_{i\ge 1}(1-x_{i}u)^{-1}$, where we set $h_{0}=1$.
Hence, for each $\lambda\in\mbb{Y}$, we have
\begin{equation*}
	\sum_{r=0}^{\infty}h_{r}(q^{\lambda-\delta})u^{r}=\prod_{i\ge 1}\frac{1}{1-t^{\lambda_{i}-i}u}=\prod_{i\ge 1}\frac{(t^{\lambda_{i}-i+1}u;t)_{\infty}}{(t^{\lambda_{i}-i}u;t)_{\infty}}
	=(t^{\lambda_{1}}u;t)_{\infty}\prod_{i\ge 1}\frac{(t^{\lambda_{i+1}-i}u;t)_{\infty}}{(t^{\lambda_{i}-i}u;t)_{\infty}}.
\end{equation*}
Here, each factor of the product $\frac{(t^{\lambda_{i+1}-i}u;t)_{\infty}}{(t^{\lambda_{i}-i}u;t)_{\infty}}=(t^{\lambda_{i+1}-i}u;t)_{\lambda_{i}-\lambda_{i+1}}$ is unity unless $\lambda_{i}>\lambda_{i+1}$. Hence,
\begin{equation*}
	\sum_{r=0}^{\infty}h_{r}(q^{\lambda-\delta})u^{r}=\prod_{i\ge 1}(1-t^{-\pr{\lambda}_{i}+i-1}u)=\sum_{r=0}^{\infty}e_{r}(t^{-\pr{\lambda}+\delta-1})(-u)^{r}
\end{equation*}
implying the desired identities.
\end{proof}

\section{Applications}
\label{sect:expectation}

In this section, we give applications of Theorem \ref{thm:correlation_operator}. We have considered four series of operators; $E_{r}$, $E_{r}(q^{-1},t^{-1})$, $G_{r}$, and $G_{r}(q^{-1},t^{-1})$, $r=1,2,\dots$, which are diagonalized by the Macdonald symmetric functions. Correspondingly, we have four series of observables for the Macdonald process, for which we compute the correlation functions below.

\subsection{Preliminaries}
As preliminaries, we introduce an expectation value of a generating function of random variables and fix the notion of the formal Fredholm determinant.
\begin{defn}
\label{defn:expectation_generating_function}
Let $f_{n}:\mbb{Y}\to\mbb{F}$, $n=0,1,\dots$ be random variables
and let $F(u)=\sum_{n=0}^{\infty}f_{n}u^{n}$ with $u$ being a formal variable be a generating function of them.
Then the expectation value of $F(u)$ with respect to the Macdonald measure is given by
\begin{equation*}
	\mbb{E}_{q,t}[F(u)]:=\sum_{n=0}^{\infty}\mbb{E}_{q,t}[f_{n}]u^{n}\in\mbb{F}[[u]].
\end{equation*}
\end{defn}

\begin{defn}
\label{defn:Fredholm_determinant}
Let $K(z,w)\in \mbb{F}[[z,z^{-1},w,w^{-1}]]$ be a formal power series in two variables
and let $u$ be another formal variable.
Then we define the formal Fredholm determinant $\det (I+uK)\in \mbb{F}[[u]]$ by
\begin{equation*}
	\det (I+uK):=1+\sum_{r=1}^{\infty}\frac{u^{r}}{r!}\int \left(\prod_{i=1}^{r}\frac{dz_{i}}{2\pi\sqrt{-1}}\right)\det [K(z_{i},z_{j})]_{1\le i,j\le r}
\end{equation*}
if it exists.
\end{defn}

\subsection{The first series observables}
\label{subsect:first_series}
The first series of observables we consider is
\begin{equation*}
	\mcal{E}_{r}:\mbb{Y}\to\mbb{F};\ \ \mcal{E}_{r}(\lambda):=e_{r}(q^{\lambda}t^{-\delta+1}),\ \ r=1,2,\dots,
\end{equation*}
Let us also introduce their generating function as follows.
\begin{equation*}
	E^{\mcal{E}}(u)=E^{\mcal{E}}(\cdot,u):=\sum_{r=0}^{\infty}\mcal{E}_{r}(\cdot)u^{r};\ \ E^{\mcal{E}}(\lambda,u)=\prod_{i\ge 1}(1+q^{\lambda_{i}}t^{-i+1}u),\ \ \lambda\in\mbb{Y},
\end{equation*}
with $\mcal{E}_{0}=1$.
Note that these observables are different from those represented by the same symbols in \cite{BorodinCorwinGorinShakirov2016}, which are the same as the observables $\pr{\mcal{E}}_{r}$, $r=1,2,\dots$ in the next Subsection \ref{subsect:second_series}.

\begin{thm}
\label{thm:correlation_sigma}
Let $N\in\{1,2,\dots\}$ and $r_{1},\dots, r_{N}\in \{1,2,\dots\}$.
The correlation function of $\mcal{E}_{r_{1}},\dots, \mcal{E}_{r_{N}}$ with respect to the $N$-step Macdonald process becomes
\begin{align*}
	&\mbb{E}^{N}_{q,t}[\mcal{E}_{r_{1}}[1]\cdots \mcal{E}_{r_{N}}[N]]\\
	&=\frac{1}{\prod_{\alpha=1}^{N}r_{\alpha}!}\int \prod_{\alpha=1}^{N}\left(\prod_{i=1}^{r_{\alpha}}\frac{dw^{(\alpha)}_{i}}{2\pi\sqrt{-1}}\right)\det\left(\frac{1}{w^{(\alpha)}_{i}-t^{-1}w^{(\alpha)}_{j}}\right)_{1\le i,j\le r_{\alpha}} \\
	&\hspace{50pt} \times \prod_{1\le \alpha\le\beta\le N}\left(\prod_{i=1}^{r_{\beta}}H(w^{(\beta)}_{i};X^{(\alpha)})^{-1}\right)\left(\prod_{i=1}^{r_{\alpha}}H((tw^{(\alpha)}_{i})^{-1};Y^{(\beta)})^{-1}\right) \\
	&\hspace{50pt} \times \prod_{1\le \alpha <\beta\le N}W(\bm{w}^{(\alpha)};\bm{w}^{(\beta)}).
\end{align*}
Here we set
\begin{equation*}
	H(w;X)=\prod_{i\ge 1}\frac{1-tx_{i}w}{1-x_{i}w},
\end{equation*}
and
\begin{equation*}
	W(\bm{w}^{(\alpha)};\bm{w}^{(\beta)})=\prod_{i=1}^{r_{\alpha}}\prod_{j=1}^{r_{\beta}}\frac{(1-w^{(\beta)}_{j}/w^{(\alpha)}_{i})(1-qt^{-1}w^{(\beta)}_{j}/w^{(\alpha)}_{i})}{(1-qw^{(\beta)}_{j}/w^{(\alpha)}_{i})(1-t^{-1}w^{(\beta)}_{j}/w^{(\alpha)}_{i})}.
\end{equation*}
\end{thm}

\begin{proof}
We use Theorem \ref{thm:correlation_operator}.
Since $e_{r}(q^{\lambda}t^{-\delta+1})=t^{r}e_{r}(q^{\lambda}t^{-\delta})$, we have from Theorem \ref{thm:Macdonald_determinant}, for $r=1,2,\dots$,
\begin{equation*}
	\mcal{O}(\mcal{E}_{r})=t^{r}\what{E}_{r}=\frac{1}{r!}\int \left(\prod_{i=1}^{r}\frac{dz_{i}}{2\pi\sqrt{-1}}\right)\det\left(\frac{1}{z_{i}-t^{-1}z_{j}}\right)_{1\le i,j\le r} \no{\eta (z_{1})\cdots \eta (z_{r})}.
\end{equation*}
Let $r_{1},\dots, r_{N}\in \{1,2,\dots\}$
We shall compute the unnormalized correlation function
\begin{align*}
	&\braket{0|\psi^{X^{(N)},Y^{(N)}}_{\mcal{E}_{r_{N}}}\cdots \psi^{X^{(1)},Y^{(1)}}_{\mcal{E}_{r_{1}}}|0} \\
	&=\frac{1}{\prod_{\alpha=1}^{N}r_{\alpha}!}\int \prod_{\alpha=1}^{N}\left(\prod_{i=1}^{r_{\alpha}}\frac{dz^{(\alpha)}_{i}}{2\pi\sqrt{-1}}\right)\det\left(\frac{1}{z^{(\alpha)}_{i}-t^{-1}z^{(\alpha)}_{j}}\right)_{1\le i,j \le r_{\alpha}} \\
	&\hspace{50pt}\times \braket{0|\Gamma (Y^{(N)})_{+}\eta (\bm{z}^{(N)}) \Gamma (X^{(N)})_{-}\cdots \Gamma (Y^{(1)})_{+}\eta (\bm{z}^{(1)}) \Gamma (X^{(1)})_{-}|0},
\end{align*}
where we set $\eta(\bm{z}^{(\alpha)})=\no{\eta(z^{(\alpha)}_{1})\cdots \eta(z^{(\alpha)}_{r_{\alpha}})}$, $\alpha=1,\dots, N$.
To compute the matrix element, we write
\begin{equation*}
	\eta(\bm{z}^{(\alpha)})=\eta(\bm{z}^{(\alpha)})_{-}\eta(\bm{z}^{(\alpha)})_{+},\ \ \alpha=1,\dots, N,
\end{equation*}
where
\begin{align*}
	\eta(\bm{z}^{(\alpha)})_{+}&:=\exp\left(-\sum_{n>0}\frac{1-t^{n}}{n}a_{n}\sum_{i=1}^{r_{\alpha}}(z^{(\alpha)}_{i})^{-n}\right),\\
	\eta(\bm{z}^{(\alpha)})_{-}&:=\exp\left(\sum_{n>0}\frac{1-t^{-n}}{n}a_{-n}\sum_{i=1}^{r_{\alpha}}(z^{(\alpha)}_{i})^{n}\right).
\end{align*}
Then we can see that
\begin{align*}
	&\Gamma (Y^{(\alpha)})_{+}\eta(\bm{z}^{(\alpha)})\Gamma (X^{(\alpha)})_{-}\\
	&=\exp\left(\sum_{n>0}\frac{1-t^{-n}}{n}p_{n}(Y^{(\alpha)})\sum_{i=1}^{r_{\alpha}}(z^{(\alpha)}_{i})^{n}\right)
		\exp\left(-\sum_{n>0}\frac{1-t^{n}}{n}p_{n}(X^{(\alpha)})\sum_{i=1}^{r_{\alpha}}(z^{(\alpha)}_{i})^{-n}\right)\\
	&\hspace{50pt}\times \eta (\bm{z}^{(\alpha)})_{-}\Gamma (Y^{(\alpha)})_{+}\Gamma (X^{(\alpha)})_{-}\eta (\bm{z}^{(\alpha)})_{+}\\
	&=\Pi (X^{(\alpha)},Y^{(\alpha)})\prod_{i=1}^{r_{\alpha}}H((z^{(\alpha)}_{i})^{-1};X^{(\alpha)})^{-1}H(t^{-1}z^{(\alpha)}_{i};Y^{(\alpha)})^{-1} \\
	&\hspace{50pt} \times \eta(\bm{z}^{(\alpha)})_{-}\Gamma (X^{(\alpha)})_{-}\Gamma (Y^{(\alpha)})_{+}\eta(\bm{z}^{(\alpha)})_{+}.
\end{align*}
Noting that the vertex operator $\eta(z)$ exhibits the following operator product expansion (OPE):
\begin{equation*}
	\eta(z)\eta(w)=\frac{(1-w/z)(1-qt^{-1}w/z)}{(1-qw/z)(1-t^{-1}w/z)}\no{\eta(z)\eta(w)}.
\end{equation*}
Then we obtain the following formula:
\begin{align*}
	&\Gamma (Y^{(\beta)})_{+}\eta (\bm{z}^{(\beta)})_{+}\eta(\bm{z}^{(\alpha)})_{-}\Gamma (X^{(\alpha)})_{-}\\
	&=\Pi (X^{(\alpha)},Y^{(\beta)})W(\bm{z}^{(\beta)},\bm{z}^{(\alpha)})\prod_{i=1}^{r_{\alpha}}H(t^{-1}z^{(\alpha)}_{i};Y^{(\beta)})^{-1}\prod_{i=1}^{r_{\beta}}H((z^{(\beta)}_{i})^{-1};X^{(\alpha)})^{-1}\\
	&\hspace{20pt}\times \eta (\bm{z}^{(\alpha)})_{-}\Gamma (X^{(\alpha)})_{-}\Gamma (Y^{(\beta)})_{+}\eta(\bm{z}^{(\beta)})_{+}.
\end{align*}

Combining the above formulas we can compute the matrix element as
\begin{align*}
	&\braket{0|\Gamma (Y^{(N)})_{+}\eta (\bm{z}^{(N)}) \Gamma (X^{(N)})_{-}\cdots \Gamma (Y^{(1)})_{+}\eta (\bm{z}^{(1)}) \Gamma (X^{(1)})_{-}|0}\\
	&=\prod_{1\le i\le j\le N}\Pi (X^{(i)},Y^{(j)})
		\prod_{1\le \alpha\le \beta\le N}\left(\prod_{i=1}^{r_{\alpha}}H(t^{-1}z^{(\alpha)}_{i};Y^{(\beta)})^{-1}\right)\left(\prod_{i=1}^{r_{\beta}}H((z^{(\beta)}_{i})^{-1};X^{(\alpha)})^{-1}\right)\\
	&\hspace{60pt}\times\prod_{1\le \alpha<\beta\le N}W(\bm{z}^{(\beta)},\bm{z}^{(\alpha)}).
\end{align*}
Therefore, we have
\begin{align*}
	&\mbb{E}^{N}_{q,t}[\mcal{E}_{r_{1}}[1]\cdots \mcal{E}_{r_{N}}[N]] \\
	&=\frac{1}{\prod_{i=1}^{N}r_{i}!}\int \prod_{\alpha=1}^{N}\left(\prod_{i=1}^{r_{\alpha}}\frac{dz^{(\alpha)}_{i}}{2\pi\sqrt{-1}}\right)
		\det\left(\frac{1}{z^{(\alpha)}_{i}-t^{-1}z^{(\alpha)}_{j}}\right)_{1\le i,j\le r_{\alpha}} \\
	&\hspace{60pt}\times\prod_{1\le \alpha\le \beta\le N}\left(\prod_{i=1}^{r_{\alpha}}H(t^{-1}z^{(\alpha)}_{i};Y^{(\beta)})^{-1}\right)\left(\prod_{i=1}^{r_{\beta}}H((z^{(\beta)}_{i})^{-1};X^{(\alpha)})^{-1}\right)\\
	&\hspace{60pt}\times\prod_{1\le \alpha<\beta\le N}W(\bm{z}^{(\beta)},\bm{z}^{(\alpha)}).
\end{align*}
Finally we adopt a transformation of integral variables so that $w^{(\alpha)}_{i}=(z^{(\alpha)}_{i})^{-1}$, $i=1,\dots, r_{\alpha}$, $\alpha=1,\dots, N$.
With help of Lemma \ref{lem:measure_invariance} and the property
\begin{equation*}
	W(z_{1}^{-1},\dots, z_{m}^{-1};w_{1}^{-1},\dots, w_{n}^{-1})=W(w_{1},\dots, w_{n};z_{1},\dots, z_{m}),
\end{equation*}
we obtain
\begin{align*}
	&\mbb{E}^{N}_{q,t}[\mcal{E}_{r_{1}}[1]\cdots \mcal{E}_{r_{N}}[N]] \\
	&=\frac{1}{\prod_{i=1}^{N}r_{i}!}\int \prod_{\alpha=1}^{N}\left(\prod_{i=1}^{r_{\alpha}}\frac{dw^{(\alpha)}_{i}}{2\pi\sqrt{-1}}\right)
		\det\left(\frac{1}{w^{(\alpha)}_{i}-t^{-1}w^{(\alpha)}_{j}}\right)_{1\le i,j\le r_{\alpha}} \\
	&\hspace{60pt}\times\prod_{1\le \alpha\le \beta\le N}\left(\prod_{i=1}^{r_{\beta}}H(w^{(\alpha)}_{i};X^{(\alpha)})^{-1}\right)\left(\prod_{i=1}^{r_{\alpha}}H((tw^{(\alpha)}_{i})^{-1};Y^{(\beta)})^{-1}\right)\\
	&\hspace{60pt}\times\prod_{1\le \alpha<\beta\le N}W(\bm{w}^{(\alpha)},\bm{w}^{(\beta)}).
\end{align*}
Then the proof is complete.
\end{proof}

In particular, when $N=1$, we have the following.
\begin{cor}
\label{cor:Fredholm_sigma}
Set
\begin{equation*}
	K^{\mcal{E}}(z,w):=\frac{1}{z-t^{-1}w}H(z;X)^{-1}H((tw)^{-1};Y)^{-1}.
\end{equation*}
Then we have
\begin{equation}
\label{eq:Fredholm_expectation_sigma}
	\mbb{E}_{q,t}[E^{\mcal{E}}(u)]=\det (I+uK^{\mcal{E}}).
\end{equation}
\end{cor}
\begin{proof}
When $N=1$, Thorem \ref{thm:correlation_sigma} reduces to
\begin{align*}
	\mbb{E}_{q,t}[\mcal{E}_{r}]=\frac{1}{r!}\int \left(\prod_{i=1}^{r}\frac{dw_{i}}{2\pi\sqrt{-1}}\right)\det [K^{\mcal{E}}(w_{i},w_{j})]_{1\le i,j\le r},\ \ r=1,2,\dots.
\end{align*}
Therefore, the desired result follows from Definitions \ref{defn:expectation_generating_function} and \ref{defn:Fredholm_determinant}.
\end{proof}

\begin{rem}
In the case that $r_{1}=\cdots =r_{N}=1$, the result of Theorem \ref{thm:correlation_sigma}
is essentially equivalent to that of \cite[Theorem 6.1]{BorodinCorwinGorinShakirov2016}, where the authors considered a random variable $\what{\mcal{E}}_{1}$ defined by
\begin{equation}
\label{eq:first_Macdonald_eigenvalue_conventional}
	\what{\mcal{E}}_{1}(\lambda)=1+(1-t)\sum_{i\ge 1}(1-q^{\lambda_{i}})t^{-i}=(t-1)e_{1}(q^{\lambda}t^{-\delta}),\ \ \lambda\in\mbb{Y}.
\end{equation}
Hence, the corresponding operator is just $\mcal{O}(\what{\mcal{E}}_{1})=\int \frac{dz}{2\pi \sqrt{-1}z}\eta (z)$.
\end{rem}

\begin{rem}
The left hand side of Eq. (\ref{eq:Fredholm_expectation_sigma}) is intrinsically dependent on two parameters $q$ and $t$,
but in the right hand side of Eq. (\ref{eq:Fredholm_expectation_sigma}), it is manifest that it is independent of $q$. Note that this $q$-independence is equivalent to that observed in \cite[Chapter VI, Section 3]{Macdonald1995}.
\end{rem}

\subsection{The second series of observables}
\label{subsect:second_series}
The next observables we consider are defined by
\begin{equation*}
	\pr{\mcal{E}}_{r}:\mbb{Y}\to\mbb{F};\ \ \pr{\mcal{E}}_{r}(\lambda):=e_{r}(q^{-\lambda}t^{\delta-1}),\ \ r=1,2,\dots.
\end{equation*}
They are just obtained from $\mcal{E}_{r}$, $r=1,2,\dots$ by inverting parameters $q$ and $t$ of the values.
We also write the generating function of them with a formal variable $u$ as
\begin{equation*}
	E^{\pr{\mcal{E}}}(u)=E^{\pr{\mcal{E}}}(\cdot,u)=\sum_{r=0}^{\infty}\pr{\mcal{E}}_{r}(\cdot)u^{r};\ \ E^{\pr{\mcal{E}}}(\lambda,u)=\prod_{i\ge 1}(1+q^{-\lambda_{i}}t^{i-1}u),\ \ \lambda\in\mbb{Y},
\end{equation*}
with the convention $\pr{\mcal{E}}_{0}=1$.
As was mentioned in Subsection \ref{subsect:first_series}, the observables $\pr{\mcal{E}}_{r}$, $r=1,2,\dots$ are the same as the observables written as $\mcal{E}_{r}$, $r=1,2,\dots$ in \cite{BorodinCorwinGorinShakirov2016}. In fact, we recover \cite[Theorem 1.1]{BorodinCorwinGorinShakirov2016} in our free field approach as follows.

\begin{thm}
\label{thm:correlation_rho}
Let $N\in\{1,2,\dots\}$ and $r_{1},\dots, r_{N}\in \{1,2,\dots\}$.
The correlation function of $\pr{\mcal{E}}_{r_{1}},\dots, \pr{\mcal{E}}_{r_{N}}$ with respect to the $N$-step Macdonald process becomes
\begin{align*}
	&\mbb{E}^{N}_{q,t}[\pr{\mcal{E}}_{r_{1}}[1]\cdots \pr{\mcal{E}}_{r_{N}}[N]]\\
	&=\frac{1}{\prod_{\alpha=1}^{N}r_{\alpha}!}\int \prod_{\alpha=1}^{N}\left(\prod_{i=1}^{r_{\alpha}}\frac{dw^{(\alpha)}_{i}}{2\pi\sqrt{-1}}\right)\det\left(\frac{1}{w^{(\alpha)}_{i}-tw^{(\alpha)}_{j}}\right)_{1\le i,j\le r_{\alpha}} \\
	&\hspace{50pt} \times \prod_{1\le \alpha\le\beta\le N}\left(\prod_{i=1}^{r_{\beta}}H(w^{(\beta)}_{i};X^{(\alpha)})\right)\left(\prod_{i=1}^{r_{\alpha}}H((qw^{(\alpha)}_{i})^{-1};Y^{(\beta)})\right) \\
	&\hspace{50pt} \times \prod_{1\le \alpha <\beta\le N}\widetilde{W}(\bm{w}^{(\alpha)};\bm{w}^{(\beta)}).
\end{align*}
Here we set
\begin{equation*}
	\widetilde{W}(\bm{w}^{(\alpha)};\bm{w}^{(\beta)})=\prod_{i=1}^{r_{\alpha}}\prod_{j=1}^{r_{\beta}}\frac{(1-w^{(\beta)}_{j}/w^{(\alpha)}_{i})(1-q^{-1}tw^{(\beta)}_{j}/w^{(\alpha)}_{i})}{(1-q^{-1}w^{(\beta)}_{j}/w^{(\alpha)}_{i})(1-tw^{(\beta)}_{j}/w^{(\alpha)}_{i})}.
\end{equation*}
\end{thm}

\begin{proof}
The proof is very similar to that of Theorem \ref{thm:correlation_sigma}.
The observable $\pr{\mcal{E}}_{r}$ is just obtained from $\mcal{E}_{r}$ by simultaneously inverting the parameters $q$ and $t$.
Since $P_{\lambda}(X;q,t)=P_{\lambda}(X;q^{-1},t^{-1})$ \cite[Chapter VI, (4.14.iv) ]{Macdonald1995}, and from Theorem \ref{thm:others_determinant}, we have
\begin{align*}
	\mcal{O}(\pr{\mcal{E}}_{r})&=t^{-r}\what{E}_{r}(q^{-1},t^{-1})\\
	&=\frac{1}{r!}\int\left(\prod_{i=1}^{r}\frac{dz_{i}}{2\pi\sqrt{-1}}\right)\det\left(\frac{1}{z_{i}-tz_{j}}\right)_{1\le i,j\le r}\no{\xi(z_{1})\cdots \xi(z_{r})}.
\end{align*}

For $r_{1},\dots, r_{N}\in \{1,2,\dots\}$, the unnormalized correlation function reads
\begin{align*}
	&\braket{0|\psi^{X^{(N)},Y^{(N)}}_{\pr{\mcal{E}}_{r_{N}}}\cdots \psi^{X^{(1)},Y^{(1)}}_{\pr{\mcal{E}}_{r_{1}}}|0} \\
	&=\frac{1}{\prod_{\alpha=1}^{N}r_{\alpha}!}\int \prod_{\alpha=1}^{N}\left(\prod_{i=1}^{r_{\alpha}}\frac{dz^{(\alpha)}_{i}}{2\pi\sqrt{-1}}\right)\det\left(\frac{1}{z^{(\alpha)}_{i}-tz^{(\beta)}_{j}}\right)_{1\le i,j \le r_{\alpha}} \\
	&\hspace{50pt}\times \braket{0|\Gamma (Y^{(N)})_{+}\xi (\bm{z}^{(N)}) \Gamma (X^{(N)})_{-}\cdots \Gamma (Y^{(1)})_{+}\xi (\bm{z}^{(1)}) \Gamma (X^{(1)})_{-}|0},
\end{align*}
where we set $\xi (\bm{z}^{(\alpha)})=\no{\xi (z^{(\alpha)}_{1})\cdots \xi (z^{(\alpha)}_{r_{\alpha}})}$, $\alpha =1,\dots, N$.
We again write $\xi (\bm{z}^{(\alpha)})=\xi (\bm{z}^{(\alpha)})_{-}\xi (\bm{z}^{(\alpha)})_{+}$, $\alpha=1,\dots, N$, where
\begin{align*}
	\xi (\bm{z}^{(\alpha)})_{+}&=\exp\left(\sum_{n>0}\frac{1-t^{n}}{n}(t/q)^{n/2}a_{n}\sum_{i=1}^{r_{\alpha}}(z^{(\alpha)}_{i})^{-n}\right), \\
	\xi (\bm{z}^{(\alpha)})_{-}&=\exp\left(-\sum_{n>0}\frac{1-t^{-n}}{n}(t/q)^{n/2}a_{-n}\sum_{i=1}^{r_{\alpha}}(z^{(\alpha)}_{i})^{n}\right).
\end{align*}
To compute the matrix element, we first note the following formula:
\begin{align*}
	&\Gamma (Y^{(\alpha)})_{+}\xi (\bm{z}^{(\alpha)}) \Gamma (X^{(\alpha)})_{-} \\
	&=\exp\left(-\sum_{n>0}\frac{1-t^{-n}}{n}(t/q)^{n/2}p_{n}(Y^{(\alpha)})\sum_{i=1}^{r_{\alpha}}(z^{(\alpha)}_{i})^{n}\right)\\
	&\hspace{20pt}\times \exp\left(\sum_{n>0}\frac{1-t^{n}}{n}(t/q)^{n/2}p_{n}(X^{(\alpha)})\sum_{i=1}^{r_{\alpha}}(z^{(\alpha)}_{i})^{-n}\right) \\
	&\hspace{20pt}\times \xi (\bm{z}^{(\alpha)})_{-}\Gamma (Y^{(\alpha)})_{+}\Gamma (X^{(\alpha)})_{-}\xi (\bm{z}^{(\alpha)})_{+} \\
	&=\Pi (X^{(\alpha)},Y^{(\alpha)}) \prod_{i=1}^{r_{\alpha}}H(t^{1/2}q^{-1/2}(z^{(\alpha)}_{i})^{-1};X^{(\alpha)})H(t^{-1/2}q^{-1/2}z^{(\alpha)}_{i};Y^{(\alpha)})\\
	&\hspace{20pt}\times \xi (\bm{z}^{(\alpha)})_{-}\Gamma (X^{(\alpha)})_{-} \Gamma (Y^{(\alpha)})_{+}\xi (\bm{z}^{(\alpha)})_{+}.
\end{align*}
From the OPE
\begin{equation*}
	\xi (z)\xi(w)=\frac{(1-w/z)(1-q^{-1}tw/z)}{(1-q^{-1}w/z)(1-tw/z)}\no{\xi (z)\xi (w)},
\end{equation*}
we also have
\begin{align*}
	&\Gamma (Y^{(\beta)})_{+}\xi (\bm{z}^{(\beta)})_{+} \xi (\bm{z}^{(\alpha)})_{-} \Gamma (X^{(\alpha)})_{-} \\
	&=\Pi (X^{(\alpha)},Y^{(\beta)}) \prod_{i=1}^{r_{\alpha}}H(t^{-1/2}q^{1/2}z^{(\alpha)}_{i};Y^{(\beta)})\prod_{i=1}^{r_{\beta}}H(t^{1/2}q^{-1/2}(z^{(\beta)}_{i})^{-1};X^{(\alpha)}) \\
	&\hspace{20pt}\times \widetilde{W}(\bm{z}^{(\beta)},\bm{z}^{(\alpha)})\xi (\bm{z}^{(\alpha)})_{-} \Gamma (X^{(\alpha)})_{-} \Gamma (Y^{(\beta)})_{+}\xi (\bm{z}^{(\beta)})_{+}.
\end{align*}

Therefore, the correlation function becomes
\begin{align*}
	&\mbb{E}_{q,t}^{N}[\pr{\mcal{E}}_{r_{1}}[1]\cdots \pr{\mcal{E}}_{r_{N}}[N]]\\
	&=\frac{1}{\prod_{\alpha=1}^{N}r_{\alpha}!}
		\int \prod_{\alpha=1}^{N}\left(\prod_{i=1}^{r_{\alpha}}\frac{dz^{(\alpha)}_{i}}{2\pi\sqrt{-1}}\right)
			\det\left(\frac{1}{z^{(\alpha)}_{i}-tz^{(\alpha)}_{j}}\right)_{1\le i,j\le r_{\alpha}} \\
	&\hspace{20pt}\times\prod_{1\le \alpha\le \beta\le N}\left(\prod_{i=1}^{r_{\beta}}H(t^{1/2}q^{-1/2}(z^{(\beta)}_{i})^{-1};X^{(\alpha)})\right)\left(\prod_{i=1}^{r_{\alpha}}H(t^{-1/2}q^{1/2}z^{(\alpha)}_{i};Y^{(\beta)})\right) \\
	&\hspace{20pt} \times \prod_{1\le \alpha <\beta\le N}\widetilde{W}(\bm{z}^{(\beta)},\bm{z}^{(\alpha)}).
\end{align*}
When we perform a transformation of integral variables so that $w^{(\alpha)}_{i}:=t^{1/2}q^{-1/2}(z^{(\alpha)}_{i})^{-1}$, $i=1,\dots, r_{\alpha}$, $\alpha=1,\dots, N$,
with help of Lemma \ref{lem:measure_invariance}, we obtain the desired result.
\end{proof}

Again, in the case when $N=1$, we obtain the following result in the same manner as for Corollary \ref{cor:Fredholm_sigma}.
\begin{cor}
Set
\begin{equation*}
	K^{\pr{\mcal{E}}}(z,w):=\frac{1}{z-tw}H(z;X)H((qw)^{-1};Y).
\end{equation*}
Then we have
\begin{equation*}
	\mbb{E}_{q,t}[E^{\pr{\mcal{E}}}(u)]=\det (I+uK^{\pr{\mcal{E}}}).
\end{equation*}
\end{cor}

\subsection{Third series of observables}
We also consider other random variables defined by
\begin{equation*}
	\mcal{G}_{r}(\lambda):=g_{r}(q^{\lambda}t^{-\delta};q,t),\ \ \lambda\in\mbb{Y},\ \ r=1,2,\dots
\end{equation*}
and their generating function
\begin{equation*}
	F^{\mcal{G}}(u)=F^{\mcal{G}}(\cdot,u)=\sum_{r=0}^{\infty}\mcal{G}_{r}(\cdot)u^{r};\ \ F^{\mcal{G}}(\lambda,u)=\prod_{i\ge 1}\frac{(q^{\lambda_{i}}t^{-i+1}u;q)_{\infty}}{(q^{\lambda_{i}}t^{-i}u;q)_{\infty}},
\end{equation*}
with $\mcal{G}_{0}=1$. Here we used the property \cite[Chapter VI, Section 2]{Macdonald1995}
\begin{equation*}
	\prod_{i\ge 1}\frac{(tx_{i}u;q)_{\infty}}{(x_{i}u;q)_{\infty}}=\exp\left(\sum_{n>0}\frac{1-t^{n}}{1-q^{n}}\frac{p_{n}(X)}{n}u^{n}\right)=1+\sum_{r=1}^{\infty}g_{r}(X;q,t)u^{r}.
\end{equation*}

\begin{thm}
\label{thm:correlation_gs}
Let $N\in\{1,2,\dots\}$ and $r_{1},\dots, r_{N}\in \{1,2,\dots\}$.
The correlation function of $\mcal{G}_{r_{1}},\dots, \mcal{G}_{r_{N}}$ with respect to the $N$-step Macdonald process becomes
\begin{align*}
	&\mbb{E}^{N}_{q,t}[\mcal{G}_{r_{1}}[1]\cdots \mcal{G}_{r_{N}}[N]]\\
	&=\frac{(-1)^{\sum_{\alpha=1}^{N}r_{\alpha}}}{\prod_{\alpha=1}^{N}r_{\alpha}!}\int \prod_{\alpha=1}^{N}\left(\prod_{i=1}^{r_{\alpha}}\frac{dw^{(\alpha)}_{i}}{2\pi\sqrt{-1}}\right)\det\left(\frac{1}{w^{(\alpha)}_{i}-qw^{(\alpha)}_{j}}\right)_{1\le i,j\le r_{\alpha}} \\
	&\hspace{50pt} \times \prod_{1\le \alpha\le\beta\le N}\left(\prod_{i=1}^{r_{\beta}}H(w^{(\beta)}_{i};X^{(\alpha)})^{-1}\right)\left(\prod_{i=1}^{r_{\alpha}}H((tw^{(\alpha)}_{i})^{-1};Y^{(\beta)})^{-1}\right) \\
	&\hspace{50pt} \times \prod_{1\le \alpha <\beta\le N}W(\bm{w}^{(\alpha)};\bm{w}^{(\beta)}).
\end{align*}
\end{thm}

\begin{proof}
From Theorems \ref{thm:operator_G} and \ref{thm:others_determinant}, we have
\begin{equation*}
	\mcal{O}(\mcal{G}_{r})=\what{G}_{r}=\frac{(-1)^{r}}{r!}\int \left(\prod_{i=1}^{r}\frac{dz_{i}}{2\pi \sqrt{-1}}\right)\det\left(\frac{1}{z_{i}-qz_{j}}\right)_{1\le i,j\le r}\no{\eta(z_{1})\cdots \eta(z_{r})}.
\end{equation*}
The essential part of the proof is computation of the matrix element
\begin{equation*}
\braket{0|\Gamma (Y^{(N)})_{+}\eta (\bm{z}^{(N)}) \Gamma (X^{(N)})_{-}\cdots \Gamma (Y^{(1)})_{+}\eta (\bm{z}^{(1)}) \Gamma (X^{(1)})_{-}|0},
\end{equation*}
but this was given in the proof of Theorem \ref{thm:correlation_sigma}.
Therefore, we obtain the desire result.
\end{proof}

In case of $N=1$, we have
\begin{cor}
Set
\begin{equation*}
	K^{\mcal{G}}(z,w):=\frac{1}{z-qw}H(z;X)^{-1}H((tw)^{-1};Y)^{-1}.
\end{equation*}
Then we have
\begin{equation*}
	\mbb{E}_{q,t}[F^{\mcal{G}}(u)]=\det (I-uK^{\mcal{G}}).
\end{equation*}
\end{cor}

\subsection{Fourth series of observables}
The final series of observables $\pr{\mcal{G}}_{r}$, $r=1,2,\dots$ is defined by
\begin{equation*}
	\pr{\mcal{G}}_{r}(\lambda)=q^{r}g_{r}(q^{-\lambda}t^{\delta-1};q,t),\ \ \lambda\in\mbb{Y},\ \ r=1,2,\dots.
\end{equation*}
Their generating function reads
\begin{equation*}
	F^{\pr{\mcal{G}}}(u)=F^{\pr{\mcal{G}}}(\cdot,u)=\sum_{r=0}^{\infty}\pr{\mcal{G}}_{r}(\cdot)u^{r};\ \ F^{\pr{\mcal{G}}}(\lambda,u)=\prod_{i\ge 1}\frac{(q^{-\lambda_{i}+1}t^{i}u;q)_{\infty}}{(q^{-\lambda_{i}+1}t^{i-1}u;q)_{\infty}},
\end{equation*}
with $\pr{\mcal{G}}_{0}=1$.

\begin{thm}
\label{thm:correlation_grho}
Let $N\in\{1,2,\dots\}$ and $r_{1},\dots, r_{N}\in \{1,2,\dots\}$.
The correlation function of $\pr{\mcal{G}}_{r_{1}},\dots, \pr{\mcal{G}}_{r_{N}}$ with respect to the $N$-step Macdonald process becomes
\begin{align*}
	&\mbb{E}^{N}_{q,t}[\pr{\mcal{G}}_{r_{1}}[1]\cdots \pr{\mcal{G}}_{r_{N}}[N]]\\
	&=\frac{(-1)^{\sum_{\alpha=1}^{N}r_{\alpha}}}{\prod_{\alpha=1}^{N}r_{\alpha}!}\int \prod_{\alpha=1}^{N}\left(\prod_{i=1}^{r_{\alpha}}\frac{dw^{(\alpha)}_{i}}{2\pi\sqrt{-1}}\right)\det\left(\frac{1}{w^{(\alpha)}_{i}-q^{-1}w^{(\alpha)}_{j}}\right)_{1\le i,j\le r_{\alpha}} \\
	&\hspace{50pt} \times \prod_{1\le \alpha\le\beta\le N}\left(\prod_{i=1}^{r_{\beta}}H(w^{(\beta)}_{i};X^{(\alpha)})\right)\left(\prod_{i=1}^{r_{\alpha}}H((qw^{(\alpha)}_{i})^{-1};Y^{(\beta)})\right) \\
	&\hspace{50pt} \times \prod_{1\le \alpha <\beta\le N}\widetilde{W}(\bm{w}^{(\alpha)};\bm{w}^{(\beta)}).
\end{align*}
\end{thm}

\begin{proof}
The operators $\what{G}_{r}(q^{-1},t^{-1})$, $r=1,2,\dots$ are diagonalized by $\ket{P_{\lambda}}$, $\lambda\in\mbb{Y}$ so that
\begin{equation*}
	\what{G}_{r}(q^{-1},t^{-1})\ket{P_{\lambda}}=g_{r}(q^{-\lambda}t^{\delta};q^{-1},t^{-1})\ket{P_{\lambda}}.
\end{equation*}
Recall that, for $r=1,2,\dots$, we have \cite[Chapter VI, (4.14) (iv)]{Macdonald1995}
\begin{equation*}
	g_{r}(X;q^{-1},t^{-1})=(qt^{-1})^{r}g_{r}(X;q,t).
\end{equation*}
Hence, we have $g_{r}(q^{-\lambda}t^{\delta};q^{-1},t^{-1})=q^{r}g_{r}(q^{-\lambda}t^{\delta-1}q,t)$, $r=1,2,\dots$, which implies that
\begin{equation*}
	\mcal{O}(\pr{\mcal{G}}_{r})=\what{G}_{r}(q^{-1},t^{-1}).
\end{equation*}
Using the expression in Theorem \ref{thm:others_determinant} and following computation in
the proof of Theorem \ref{thm:correlation_rho},
we obtain the formula in Theorem \ref{thm:correlation_grho}.
\end{proof}

In case of $N=1$, we have
\begin{cor}
\label{cor:fredholm_grho}
Set
\begin{equation*}
	K^{\pr{\mcal{G}}}(z,w):=\frac{1}{z-q^{-1}w}H(z;X)H((qw)^{-1};Y).
\end{equation*}
Then we have
\begin{equation*}
	\mbb{E}_{q,t}[F^{\pr{\mcal{G}}}(u)]=\det (I-uK^{\pr{\mcal{G}}}).
\end{equation*}
\end{cor}

Corollary \ref{cor:fredholm_grho} admits a nontrivial $q$-Whittaker ($t\to 0$) limit.
The generating function $F^{\pr{\mcal{G}}}(u)$ reduces at the $q$-Whittaker limit to
\begin{equation*}
	F^{\pr{\mcal{G}}}(\lambda,u)\big|_{t=0}=\frac{1}{(q^{-\lambda_{1}+1}u;q)_{\infty}},\ \ \lambda\in\mbb{Y},
\end{equation*}
and the kernel function becomes
\begin{equation}
\label{eq:kernel_function_q-Whittaker}
	K^{\pr{\mcal{G}}}(z,w)\big|_{t=0}=\frac{1}{z-q^{-1}w}\prod_{i\ge 1}\frac{1}{1-x_{i}z}\frac{1}{1-y_{i}/(qz)}.
\end{equation}
Let us informally state the result at the $q$-Whittaker limit:
\begin{cor}
\label{cor:q-Whittaker}
At the $q$-Whittaker limit $t\to 0$, we have
\begin{equation*}
	\mbb{E}_{q,0}\left[\frac{1}{(q^{-\lambda_{1}+1}u;q)_{\infty}}\right]=\det \left(I-uK^{\pr{\mcal{G}}}\big|_{t=0}\right).
\end{equation*}
\end{cor}

\begin{rem}
A similar observable has been considered in \cite[Theorem 3.3]{BorodinCorwinFerrariVeto2015} (and in \cite[Theorem 3.20]{BarraquandBorodinCorwin2020} for a half-space counterpart).
In comparison with these results, our result does not require any specialization of variables
and the kernel function (\ref{eq:kernel_function_q-Whittaker}) seems to be simpler, though its good application has not yet been found.
\end{rem}

\section{Generalized Macdonald measure}
\label{sect:gen_Mac_meas}
In this section, we propose a generalization of the Macdonald measure using the representation theory of the DIM algebra \cite{DingIohara1997,Miki2007} and give a proof of Theorem \ref{thm:generalized_Mcdonald_expectation}.

\subsection{Ding--Iohara--Miki algebra}
We begin with introducing the DIM algebra following \cite{DingIohara1997,AwataFeiginHoshinoKanaiShiraishiYanagida2011}.
Let us introduce a formal power series
\begin{equation*}
	g(z)=\frac{G^{+}(z)}{G^{-}(z)}\in\mbb{F}[[z]],\ \ G^{\pm}(z)=(1-q^{\pm 1}z)(1-t^{\mp 1}z)(1-q^{\mp 1}t^{\pm 1}z).
\end{equation*}
The DIM algebra $\mcal{U}$ is a unital associative algebra over $\mbb{F}$ generated by currents
\begin{equation*}
	x^{\pm}(z)=\sum_{n\in\mbb{Z}}x^{\pm}_{n}z^{-n},\ \ \psi^{\pm}(z)=\sum_{\pm n\in\mbb{Z}_{\ge 0}}\psi^{\pm}_{n}z^{-n}
\end{equation*}
and an invertible central element $\gamma^{1/2}$ subject to relations
\begin{align*}
	\psi^{\pm}(z)\psi^{\pm}(w)&=\psi^{\pm}(w)\psi^{\pm}(z), \\
	\psi^{+}(z)\psi^{-}(w)&=\frac{g(\gamma w/z)}{g(\gamma^{-1} w/z)}\psi^{-}(w)\psi^{+}(z), \\
	\psi^{+}(z)x^{\pm}(w)&=g(\gamma^{\mp 1/2}w/z)^{\mp 1}x^{\pm}(w)\psi^{+}(z), \\
	\psi^{-}(z)x^{\pm}(w)&=g(\gamma^{\mp 1/2}w/z)^{\pm 1}x^{\pm}(w)\psi^{-}(z),\\
	[x^{+}(z),x^{-}(w)]&=\frac{(1-q)(1-t^{-1})}{1-qt^{-1}}(\delta(\gamma^{-1}z/w)\psi^{+}(\gamma^{1/2}w)-\delta (\gamma z/w)\psi^{-}(\gamma^{-1/2}w)), \\
	G^{\mp}(z/w)x^{\pm}(z)x^{\pm}(w)&=G^{\pm}(z/w)x^{\pm}(w)x^{\pm}(z).
\end{align*}
Here we set $\delta(z)=\sum_{n\in\mbb{Z}}z^{n}$ as the formal delta distribution.

The DIM algebra $\mcal{U}$ is a formal Hopf algebra with coproduct $\Delta$ defined by
\begin{align*}
	\Delta (\gamma^{1/2})&=\gamma^{1/2}\otimes \gamma^{1/2},\\
	\Delta (\psi^{\pm}(z))&=\psi^{\pm}(\gamma^{\pm 1/2}_{(2)}z)\otimes \psi^{\pm}(\gamma^{\mp 1/2}_{(1)}z), \\
	\Delta (x^{+}(z))&=x^{+}(z)\otimes 1 +\psi^{1}(\gamma^{1/2}_{(1)}z)\otimes x^{+}(\gamma_{(1)}z), \\
	\Delta (x^{-}(z))&=x^{-}(\gamma_{(2)}z)\otimes \psi^{+}(\gamma^{1/2}_{(2)}z)+1\otimes x^{-}(z).
\end{align*}
Here we used the notation $\gamma^{1/2}_{(1)}=\gamma^{1/2}\otimes 1$ and $\gamma^{1/2}_{(2)}=1\otimes \gamma^{1/2}$.

To consider a representation of $\mcal{U}$ on a Fock space $\widetilde{\mcal{F}}:=\mbb{C}(q^{1/4},t^{1/4})\otimes_{\mbb{F}}\mcal{F}$,
we introduce, in addition to the vertex operators $\eta(z)$ and $\xi (z)$, the following ones
\begin{equation*}
	\varphi^{\pm}(z):=\exp\left(\mp \sum_{n>0}\frac{1-t^{\pm n}}{n}(1-(t/q)^{n})(t/q)^{-n/4}a_{\pm n}z^{\mp n}\right).
\end{equation*}
Then the assignment $\rho:\mcal{U}\to\mrm{End}(\widetilde{\mcal{F}})$ defined by
\begin{equation*}
	\rho (\gamma^{1/2})=(t/q)^{1/4},\ \ \rho (\psi^{\pm}(z))=\varphi^{\pm}(z),\ \ \rho (x^{+}(z))=\eta (z),\ \ \rho (x^{-}(z))=\xi (z)
\end{equation*}
gives a {\it level one} representation of $\mcal{U}$ on $\widetilde{\mcal{F}}$.
As we saw in Theorem \ref{thm:Macdonald_operators_free_field}, the zero mode $x_{0}^{+}$ acts essentially as the first Macdonald operator on $\mcal{F}$ so that
\begin{equation*}
	\rho (x_{0}^{+})\ket{P_{\lambda}}=\what{\mcal{E}}_{1}(\lambda)\ket{P_{\lambda}}, \ \ \lambda\in\mbb{Y}.
\end{equation*}
See (\ref{eq:first_Macdonald_eigenvalue_conventional}) for definition of the eigenvalues.

Using the formal Hopf algebra structure of $\mcal{U}$, we can equip the $m$-fold tensor product $\widetilde{\mcal{F}}^{\otimes m}$, $m\in\mbb{N}$ with an action of $\mcal{U}$.
We set $\Delta^{(1)}=\mrm{Id}$, $\Delta^{(2)}=\Delta$ and, inductively,
$\Delta^{(m)}=(\mrm{Id}\otimes\cdots\otimes \mrm{Id}\otimes \Delta)\circ \Delta^{(m-1)}$, $m\in\mbb{N}$.
Then the assignment
\begin{equation*}
	\rho^{(m)}:=(\rho\otimes \cdots\otimes \rho)\circ \Delta^{(m)}:\mcal{U}\to\mrm{End}(\widetilde{\mcal{F}}^{\otimes m})
\end{equation*}
gives an {\it level $m$} representation of $\mcal{U}$ on $\widetilde{\mcal{F}}^{\otimes m}$.

\subsection{Generalized Macdonald functions}
We write, for $T\in\mrm{End}(\widetilde{\mcal{F}})$ and $i=1,\dots, m$,
\begin{equation*}
	T^{(i)}:=\mrm{Id}\otimes\cdots \otimes \mrm{Id}\otimes \overset{i}{\check{T}}\otimes \mrm{Id}\otimes \cdots \otimes \mrm{Id},
\end{equation*}
and set
\begin{equation*}
	\bm{\Gamma} (\bm{X})_{\pm}=\prod_{i=1}^{m}\Gamma (X^{(i)})_{\pm}^{(i)},\ \ \bm{X}=(X^{(1)},\dots, X^{(m)}).
\end{equation*}
Then the isomorphisms $\iota^{\otimes m}:\widetilde{\mcal{F}}^{\otimes m}\to \widetilde{\Lambda}_{X^{(1)}}\otimes\cdots\otimes \widetilde{\Lambda}_{X^{(m)}}$
and $(\iota^{\dagger})^{\otimes m}:(\widetilde{\mcal{F}}^{\dagger})^{\otimes m}\to \widetilde{\Lambda}_{X^{(1)}}\otimes\cdots\otimes \widetilde{\Lambda}_{X^{(m)}}$,
where $\widetilde{\Lambda}=\mbb{C}(q^{1/4},t^{1/4})\otimes_{\mbb{F}}\Lambda$,
are identified with
\begin{equation*}
	\iota^{\otimes m}=\bra{\bm{0}}\bm{\Gamma}(\bm{X})_{+},\ \ (\iota^{\dagger})^{\otimes m}=\bm{\Gamma}(\bm{X})_{-}\ket{\bm{0}},
\end{equation*}
where we set $\ket{\bm{0}}=\ket{0}\otimes\cdots\otimes\ket{0}$ and $\bra{\bm{0}}=\bra{0}\otimes\cdots\otimes \bra{0}$.

We write an $m$-tuple of partitions as $\bm{\lambda}=(\lambda^{(1)},\dots, \lambda^{(m)})\in\mbb{Y}^{m}$.
An analogue of the dominance order on $\mbb{Y}^{m}$ is defined so that,
for $\bm{\lambda}$, $\bm{\mu}\in\mbb{Y}^{m}$, we say that $\bm{\lambda}\ge \bm{\mu}$ if
\begin{align*}
	|\lambda^{(1)}|+\cdots |\lambda^{(j-1)}|+\sum_{k=1}^{i}\lambda^{(j)}_{k}\ge
	|\mu^{(1)}|+\cdots +|\mu^{(j-1)}|+\sum_{k=1}^{i}\mu^{(j)}_{k}
\end{align*}
holds for all $i\ge 1$ and $j=1,\dots, m$.
For a monomial symmetric function $m_{\lambda}$, $\lambda\in\mbb{Y}$, we write $\bra{m_{\lambda}}=(\iota^{\dagger})^{-1}(m_{\lambda})$ and set
\begin{equation*}
	\bra{m_{\bm{\lambda}}}:=\bra{m_{\lambda^{(1)}}}\otimes\cdots\otimes\bra{m_{\lambda^{(m)}}}\in (\mcal{F}^{\dagger})^{\otimes m},\ \ \bm{\lambda}\in\mbb{Y}^{m}.
\end{equation*}

The following proposition was presented in \cite{AwataFeiginHoshinoKanaiShiraishiYanagida2011} (see also \cite{Ohkubo2017,FukudaOhkuboShiraishi2019, MironovMorozov2019} for recent studies).
\begin{prop}[{\cite{AwataFeiginHoshinoKanaiShiraishiYanagida2011}}]
\label{prop:char_generalized_Macdonald}
For an $m$-tuple of partitions $\bm{\lambda}\in\mbb{Y}^{m}$, a vector $\bra{P_{\bm{\lambda}}}\in (\widetilde{\mcal{F}}^{\dagger})^{\otimes m}$ is uniquely determined by
the following properties:
\begin{align*}
	\bra{P_{\bm{\lambda}}}&=\bra{m_{\bm{\lambda}}}+\sum_{\bm{\mu}<\bm{\lambda}}c_{\bm{\lambda}\bm{\mu}}\bra{m_{\bm{\mu}}},\ \ c_{\bm{\lambda}\bm{\mu}}\in\mbb{C}(q^{1/4},t^{1/4}),\\
	\bra{P_{\bm{\lambda}}}X^{+}_{0}&=\what{\mcal{E}}^{(m)}_{1}(\bm{\lambda})\bra{P_{\bm{\lambda}}},\ \ \what{\mcal{E}}^{(m)}_{1}(\bm{\lambda})=\sum_{i=1}^{m}\what{\mcal{E}}_{1}(\lambda^{(i)}),
\end{align*}
where we set $X^{+}_{0}:=\rho^{(m)}(x^{+}_{0})$.
The collection $\{\bra{P_{\bm{\lambda}}}|\bm{\lambda}\in\mbb{Y}^{m}\}$ forms a basis of $(\widetilde{\mcal{F}}^{\dagger})^{\otimes m}$.
\end{prop}

\begin{defn}
\label{def:generalized_Macdonald}
Let $\{\ket{Q_{\bm{\lambda}}}|\bm{\lambda}\in\mbb{Y}^{m}\}$ be the basis of $\widetilde{\mcal{F}}^{\otimes m}$ dual to $\{\bra{P_{\bm{\lambda}}}|\bm{\lambda}\in\mbb{Y}^{m}\}$
so that $\braket{P_{\bm{\lambda}}|Q_{\bm{\mu}}}=\delta_{\bm{\lambda}\bm{\mu}}$, $\bm{\lambda},\bm{\mu}\in\mbb{Y}^{m}$.
We also set $P_{\bm{\lambda}}:=(\iota^{\dagger})^{\otimes m}(\bra{P_{\bm{\lambda}}})\in\Lambda^{\otimes m}$
and $Q_{\bm{\lambda}}:=\iota^{\otimes m}(\ket{Q_{\bm{\lambda}}})\in\Lambda^{\otimes m}$, $\bm{\lambda}\in\mbb{Y}^{m}$.
\end{defn}

\begin{rem}
Differently from the usual Macdonald symmetric functions, $P_{\bm{\lambda}}$ and $Q_{\bm{\lambda}}$ are not proportional to each other.
\end{rem}

\begin{prop}
\label{prop:expansion_generalized}
We have the following expansions:
\begin{equation*}
	\bm{\Gamma}(\bm{X})_{-}\ket{\bm{0}}=\sum_{\bm{\lambda}\in\mbb{Y}^{m}}P_{\bm{\lambda}}(\bm{X})\ket{Q_{\bm{\lambda}}},\ \ 
	\bra{\bm{0}}\bm{\Gamma}(\bm{X})_{+}=\sum_{\bm{\lambda}\in\mbb{Y}^{m}}Q_{\bm{\lambda}}(\bm{X})\bra{P_{\bm{\lambda}}}.
\end{equation*}
\end{prop}
As a corollary of this proposition, we obtain the Cauchy-type identity:
\begin{cor}
We have
\begin{equation*}
	\sum_{\bm{\lambda}\in\mbb{Y}^{m}}P_{\bm{\lambda}}(\bm{X})Q_{\bm{\lambda}}(\bm{Y})=\Pi^{(m)}(\bm{X},\bm{Y}):=\prod_{i=1}^{m}\Pi (X^{(i)},Y^{(i)}).
\end{equation*}
\end{cor}
\begin{proof}
Computation of
\begin{equation*}
	\braket{\bm{0}|\bm{\Gamma}(\bm{Y})_{+}\bm{\Gamma}(\bm{X})_{-}|\bm{0}}
\end{equation*}
in two ways proves the desired result.
On one hand, we have
\begin{equation*}
	\braket{\bm{0}|\bm{\Gamma}(\bm{Y})_{+}\bm{\Gamma}(\bm{X})_{-}|\bm{0}}=\sum_{\bm{\lambda}\in\mbb{Y}^{m}}P_{\bm{\lambda}}(\bm{X})Q_{\bm{\lambda}}(\bm{Y}),
\end{equation*}
while, on the other hand, we also have
\begin{equation*}
	\braket{\bm{0}|\bm{\Gamma}(\bm{Y})_{+}\bm{\Gamma}(\bm{X})_{-}|\bm{0}}=\prod_{i=1}^{m}\braket{0|\Gamma (Y^{(i)})_{+}\Gamma (X^{(i)})_{-}|0}
	=\Pi^{(m)}(\bm{X},\bm{Y}).
\end{equation*}
\end{proof}

\subsection{Generalized Macdonald measure}
Now we define {\it the level $m$ generalized Macdonald measure} on $\mbb{Y}^{m}$.
\begin{defn}
The level $m$ generalized Macdonald measure is a probability measure on $\mbb{Y}^{m}$ so that
the weight of $\bm{\lambda}\in\mbb{Y}^{m}$ is given by
\begin{equation*}
	\mbb{GM}^{m}_{q,t}(\bm{\lambda}):=\frac{1}{\Pi^{(m)}(\bm{X},\bm{Y})}P_{\bm{\lambda}}(\bm{X})Q_{\bm{\lambda}}(\bm{Y}).
\end{equation*}
We write $\mbb{GE}^{m}_{q,t}$ for the expectation value under the level $m$ generalized Macdonald measure.
\end{defn}

\begin{proof}[Proof of Theorem \ref{thm:generalized_Mcdonald_expectation}]
We compute the quantity
\begin{equation*}
	\braket{\bm{0}|\bm{\Gamma}(\bm{Y})_{+}X^{+}_{0}\bm{\Gamma}(\bm{X})_{-}|\bm{0}}
\end{equation*}
in two ways.
On one hand, it is shown that
\begin{equation*}
	\braket{\bm{0}|\bm{\Gamma}(\bm{Y})_{+}X^{+}_{0}\bm{\Gamma}(\bm{X})_{-}|\bm{0}}=\Pi^{(m)}(\bm{X},\bm{Y})\mbb{GE}^{m}_{q,t}[\underline{\mcal{E}}^{(m)}]
\end{equation*}
from the definition of level $m$ generalized Macdonald functions presented in Proposition \ref{prop:char_generalized_Macdonald}, Definition \ref{def:generalized_Macdonald}, and the expansions in Proposition \ref{prop:expansion_generalized}.

On the other hand, we can also write
\begin{equation*}
	X^{+}_{0}=\int \frac{dz}{2\pi\sqrt{-1}z}X^{+}(z),\ \ X^{+}(z)=\rho^{(m)}(x^{+}(z))=\sum_{i=1}^{m}\tilde{\Lambda}_{i}(z),
\end{equation*}
where we set
\begin{equation*}
	\tilde{\Lambda}_{i}(z)=\varphi^{-}(p^{-1/4}z)\otimes \varphi^{-}(p^{-3/4}z)\otimes \cdots \otimes \varphi^{-}(p^{-(2i-3)/4}z)\otimes \overset{i}{\check{\eta}}(p^{-(i-1)/2}z)
	\otimes \mrm{Id}\otimes \cdots \otimes \mrm{Id}
\end{equation*}
for $i=1,\dots, m$.
It can be verified that $\Gamma (X)_{+}$ and $\varphi^{-}(z)$ exhibit OPE
\begin{equation*}
	\Gamma (X)_{+}\varphi^{-}(z)=\prod_{k\ge 1}\frac{(1-p^{-3/4}x_{k}z)(1-t^{-1}p^{1/4}x_{k}z)}{(1-p^{1/4}x_{k}z)(1-t^{-1}p^{-3/4}x_{k}z)}\varphi^{-}(z)\Gamma (X)_{+}.
\end{equation*}
Then a similar argument as in Section \ref{sect:expectation} gives the desire result.
\end{proof}

\appendix
\section{Proof of Theorem \ref{thm:free_field_G_inverse}}
\label{app:proof_free_field_G_inverse}
Fix $r \in\mbb{Z}_{\ge 1}$.
What we show is the identity
\begin{equation*}
	\bra{0}\Gamma (X)_{+}\what{G}_{r}(q^{-1},t^{-1})=G_{r}(q^{-1},t^{-1})\bra{0}\Gamma (X)_{+},
\end{equation*}
which is equivalent to
\begin{equation*}
	\bra{0}\Gamma _{n} (X)_{+}\what{G}_{r}(q^{-1},t^{-1})=G^{(n)}_{r}(q^{-1},t^{-1})\bra{0}\Gamma_{n} (X)_{+},\ \ n=1,2,\dots,
\end{equation*}
where
\begin{equation*}
	\Gamma_{n}(X)_{+}=\exp\left(\sum_{m>0}\frac{1-t^{m}}{1-q^{m}}\frac{p_{m}^{(n)}(X)}{m}a_{m}\right)
\end{equation*}
is the $n$-variable reduction of $\Gamma (X)_{+}$.

The following lemma can be checked by standard computation:
\begin{lem}
We have
\begin{align*}
	\Gamma_{n}(X)_{+}\xi(z)&=\prod_{i=1}^{n}\frac{1-t^{1/2}q^{-1/2}x_{i}z}{1-t^{-1/2}q^{-1/2}x_{i}z}\xi (z)\Gamma _{n}(X)_{+}.
\end{align*}
\end{lem}

By using this, we can see that
\begin{align*}
	&\bra{0}\Gamma_{n}(X)_{+}\what{G}_{r}(q^{-1},t^{-1}) \\
	&=\frac{(-1)^{r}q^{-\binom{r}{2}}}{(q;q)_{r}}\int \left(\prod_{i=1}^{r}\frac{dz_{i}}{2\pi\sqrt{-1}z_{i}}\right) \left(\prod_{1\le i<j\le r}\frac{1-z_{j}/z_{i}}{1-q^{-1}z_{j}/z_{i}}\right)\\
	&\hspace{80pt}\times\left(\prod_{i=1}^{n}\prod_{j=1}^{r}\frac{1-t^{1/2}q^{-1/2}x_{i}z_{j}}{1-t^{-1/2}q^{-1/2}x_{i}z_{j}}\right)
		\bra{0}\no{\xi(z_{1})\cdots \xi(z_{r})}\Gamma_{n}(X)_{+}
\end{align*}

As a generalization, we introduce the following object:
for $\mu=(\mu_{1},\dots, \mu_{n})\in\mbb{Z}^{n}$,
\begin{align}
\label{eq:G_general}
	&\mfrak{G}_{r}^{(\mu)}(q^{-1},t^{-1}) \\
	&:=\int \left(\prod_{i=1}^{r}\frac{dz_{i}}{2\pi\sqrt{-1}z_{i}}\right) \left(\prod_{1\le i<j\le r}\frac{1-z_{j}/z_{i}}{1-q^{-1}z_{j}/z_{i}}\right)
	\left(\prod_{i=1}^{n}\prod_{j=1}^{r}\frac{1-t^{1/2}q^{\mu_{i}-1/2}x_{i}z_{j}}{1-t^{-1/2}q^{-1/2}x_{i}z_{j}}\right) \notag \\
	&\hspace{50pt}\times\bra{0}\no{\xi(z_{1})\cdots \xi(z_{r})}\Gamma_{n}(X)_{+}. \notag
\end{align}
Correspondingly, we set
\begin{align}
\label{eq:H_general}
	&H_{r}^{(n),(\mu)}(q^{-1},t^{-1})\\
	&:=\sum_{\substack{\nu\in(\mbb{Z}_{\ge 0})^{n}\\ |\nu|=r}}\left(\prod_{1\le i<j\le n}\frac{q^{-\nu_{i}}x_{i}-q^{-\nu_{j}}x_{j}}{x_{i}-x_{j}}\right)
		\left(\prod_{i,j=1}^{n}\frac{(t^{-1}x_{i}/q^{\mu_{j}}x_{j};q^{-1})_{\nu_{i}}}{(q^{-1}x_{i}/x_{j};q^{-1})_{\nu_{i}}}\right)
		\prod_{i=1}^{n}T_{q^{-1},x_{i}}^{\nu_{i}}.\notag
\end{align}
Then Theorem \ref{thm:free_field_G_inverse} is the case when $\mu=(0,\dots, 0)$ of the following one:
\begin{thm}
\label{thm:free_field_G_general}
For $\mu\in \mbb{Z}^{n}$ and $r=1,2,\dots$, we have
\begin{align*}
	&\mfrak{G}_{r}^{(\mu)}(q^{-1},t^{-1})\\
	&=t^{rn}q^{r|\mu|}\sum_{l=0}^{r}(-1)^{l}q^{\binom{l}{2}}q^{l(r-l)}(q^{-r+l-1};q^{-1})_{l}H_{l}^{(n),(\mu)}(q^{-1},t^{-1})
		\bra{0}\Gamma_{n}(X)_{+}.
\end{align*}
\end{thm}

To prove this theorem, we prepare lemmas.
The first one can be checked by a simple calculation.
\begin{lem}
\label{lem:q-inv_shift}
We have
\begin{equation*}
	\bra{0}\xi (t^{1/2}q^{1/2}x_{k}^{-1})\Gamma _{n}(X)_{+}=T_{q^{-1},x_{k}}\bra{0}\Gamma_{n}(X)_{+}.
\end{equation*}
\end{lem}

We also have
\begin{lem}
\label{lem:partial_fraction}
Let $\nu=(\nu_{1},\dots,\nu_{n})\in\mbb{Z}_{n}$.
Then
\begin{equation*}
	\sum_{k=1}^{n}(q^{-\nu_{k}}-1)\left(\prod_{i\neq k}\frac{x_{k}-q^{-\nu_{i}}x_{i}}{x_{k}-x_{i}}\right)=q^{-|\nu|}-1.
\end{equation*}
\end{lem}
\begin{proof}
By comparing residues and behavior at $z\to \infty$, we have
\begin{equation*}
	\prod_{i=1}^{n}\frac{z-q^{-\nu_{i}}x_{i}}{z-x_{i}}=\sum_{k=1}^{n}\frac{(1-q^{-\nu_{k}})x_{k}}{z-x_{k}}\left(\prod_{i\neq k}\frac{x_{k}-q^{-\nu_{i}}x_{i}}{x_{k}-x_{i}}\right)+1.
\end{equation*}
Then, setting $z=0$, we can see the desired result.
\end{proof}

\begin{proof}[Proof of Theorem \ref{thm:free_field_G_general}]
We first integrate out the variable $z_{1}$ in Eq. (\ref{eq:G_general}).
Then the residues at $z_{1}=t^{1/2}q^{1/2}x_{i}^{-1}$, $i=1,\dots, n$ and $z_{1}=\infty$ contribute.
(To make this computation easier to see, it might be convenient to transform the variables so that $w_{i}=z_{i}^{-1}$, $i=1,\dots, r$.)
Consequently, we obtain
\begin{align*}
	&\mfrak{G}_{r}^{(\mu)}(q^{-1},t^{-1}) \\
	&=\int \left(\prod_{i=2}^{r}\frac{dz_{i}}{2\pi\sqrt{-1}z_{i}}\right)
	\sum_{k=1}^{n}(1-tq^{\mu_{k}})\left(\prod_{2\le i<j\le r}\frac{1-z_{j}/z_{i}}{1-q^{-1}z_{j}/z_{i}}\right)
		\left(\prod_{j=2}^{r}\frac{1-t^{1/2}q^{\mu_{k}-1/2}x_{k}z_{j}}{1-t^{-1/2}q^{-3/2}x_{k}z_{j}}\right) \\
	&\hspace{120pt}\times \left(\prod_{\substack{i=1\\ i\neq k}}^{n}\prod_{j=2}^{r}\frac{1-t^{1/2}q^{\mu_{i}-1/2}x_{i}z_{j}}{1-t^{-1/2}q^{-1/2}x_{i}z_{j}}\right)
		\left(\prod_{\substack{i=1\\ i\neq k}}^{n}\frac{x_{k}-tq^{\mu_{i}}x_{i}}{x_{k}-x_{i}}\right) \\
	&\hspace{120pt}\times \bra{0}\no{\xi (t^{1/2}q^{1/2}x_{k}^{-1})\xi (z_{2})\cdots \xi (z_{r})}\Gamma _{n}(X)_{+} \\
	&\hspace{20pt}+t^{n}q^{|\mu|}\mfrak{G}_{r-1}^{(\mu)}(q^{-1},t^{-1}).
\end{align*}
Here we shall use Lemma \ref{lem:q-inv_shift} to find
\begin{align*}
	\mfrak{G}_{r}^{(\mu)}(q^{-1},t^{-1})
	&=\sum_{k=1}^{n}(1-tq^{\mu_{k}})\prod_{\substack{i=1\\ i\neq k}}^{n}\frac{x_{k}-tq^{\mu_{i}}x_{i}}{x_{k}-x_{i}}T_{q^{-1},x_{k}}\mfrak{G}_{r-1}^{(\mu+\epsilon_{k})}(q^{-1},t^{-1}) \\
	&\hspace{20pt}+t^{n}q^{|\mu|}\mfrak{G}_{r-1}^{(\mu)}(q^{-1},t^{-1}),
\end{align*}
where we set $\epsilon_{k}=(0,\dots, 0, \overset{k}{\check{1}},0,\dots, 0)$.

We prove the theorem by induction in $r=1,2,\dots$.
When $r=1$, we can see that
\begin{align*}
	\mfrak{G}_{1}^{(\mu)}(q^{-1},t^{-1})
	&=\left(\sum_{k=1}^{n}(1-tq^{\mu_{k}})\prod_{\substack{i=1\\ i\neq k}}^{n}\frac{x_{k}-tq^{\mu_{i}}x_{i}}{x_{k}-x_{i}}T_{q^{-1},x_{k}}+t^{n}q^{|\mu|}\right)
		\bra{0}\Gamma _{n}(X)_{+},
\end{align*}
while, from Eq. (\ref{eq:H_general}), we also have
\begin{equation*}
	H_{1}^{(n),(\mu)}(q^{-1},t^{-1})=-t^{-n}q^{-|\mu|}\sum_{k=1}^{n}\frac{1-tq^{\mu_{k}}}{1-q^{-1}}\prod_{\substack{i=1 \\ i\neq k}}^{n}\frac{x_{k}-tq^{\mu_{i}}x_{i}}{x_{k}-x_{i}}T_{q^{-1},x_{k}}.
\end{equation*}
Therefore, it follows that
\begin{equation*}
	\mfrak{G}_{1}^{(\mu)}(q^{-1},t^{-1})
	=t^{n}q^{|\mu|}\left(1-(1-q^{-1})H_{1}^{(n),(\mu)}(q^{-1},t^{-1})\right)\bra{0}\Gamma_{n}(X)_{+},
\end{equation*}
which is the desired result of the theorem for $r=1$.

We suppose that the theorem holds at $r-1$.
To use this induction hypothesis, we make some preliminaries.
For $\nu=(\nu_{1},\dots,\nu_{n})\in (\mbb{Z}_{\ge 0})^{n}$ and 
$\mu=(\mu_{1},\dots, \mu_{n})\in\mbb{Z}^{n}$, set
\begin{equation*}
	h_{\nu}^{(\mu)}(q^{-1},t^{-1})=
	\left(\prod_{1\le i<j\le n}\frac{q^{-\nu_{i}}x_{i}-q^{-\nu_{j}}x_{j}}{x_{i}-x_{j}}\right)
		\left(\prod_{i,j=1}^{n}\frac{(t^{-1}x_{i}/q^{\mu_{j}}x_{j};q^{-1})_{\nu_{i}}}{(q^{-1}x_{i}/x_{j};q^{-1})_{\nu_{i}}}\right).
\end{equation*}
Then we have the following expression:
\begin{equation*}
	H_{r}^{(n),(\mu)}(q^{-1},t^{-1})=\sum_{\substack{\nu\in(\mbb{Z}_{\ge 0})^{n}\\ |\nu|=r}}h_{\nu}^{(\mu)}(q^{-1},t^{-1})\prod_{i=1}^{n}T_{q^{-1},x_{i}}^{\nu_{i}}.
\end{equation*}
Observe that
\begin{align}
\label{eq:assist1}
	h_{\nu}^{(\mu+\epsilon_{k})}(q^{-1},t^{-1})
	=\frac{1-t^{-1}q^{-\mu_{k}-\nu_{k}}}{1-t^{-1}q^{-\mu_{k}}}\prod_{\substack{i=1\\ i\neq k}}^{n}\frac{x_{k}-t^{-1}q^{-\mu_{k}-\nu_{i}}x_{i}}{x_{k}-t^{-1}q^{-\mu_{k}}x_{i}}h_{\nu}^{(\mu)}(q^{-1},t^{-1})
\end{align}
and
\begin{align}
\label{eq:assist2}
	&T_{q^{-1},x_{k}}h_{\nu}^{(\mu)}(q^{-1},t^{-1})\\
	&=\frac{1-q^{-\nu_{k}-1}}{1-t^{-1}q^{-\mu_{k}-\nu_{k}}}
		\prod_{\substack{i=1\\ i\neq k}}^{n}\frac{t^{-1}q^{-\mu_{k}+1}x_{i}-x_{k}}{x_{i}-t^{-1}q^{-\mu_{i}}x_{i}}\frac{x_{k}-q^{-\nu_{i}}x_{i}}{x_{k}-t^{-1}q^{-\mu_{k}-\nu_{i}+1}x_{i}}h_{\nu+\epsilon_{k}}^{(\mu)}(q^{-1},t^{-1})T_{q^{-1},x_{k}}.\notag
\end{align}

By the induction hypothesis, we have
\begin{align*}
	\mfrak{G}_{r}^{(\mu)}(q^{-1},t^{-1})
	&=\Biggl[t^{(r-1)n}q^{(r-1)(|\mu|+1)}\sum_{l=0}^{r-1}(-1)^{l}q^{\binom{l}{2}}q^{l(r-l+1)}(q^{-r+l};q^{-1})_{l} \\
	&\hspace{20pt}\times\sum_{k=1}^{n}(1-tq^{\mu_{k}})\prod_{\substack{i=1\\ i\neq k}}^{n}\frac{x_{k}-tq^{\mu_{i}}x_{i}}{x_{k}-x_{i}}T_{q^{-1},x_{k}}H_{l}^{(n),(\mu+\epsilon_{k})}(q^{-1},t^{-1}) \\
	&\hspace{20pt}+t^{rn}q^{r|\mu|}\sum_{l=0}^{r-1}(-1)^{l}q^{\binom{l}{2}}q^{l(r-l+1)}(q^{-r+l};q^{-1})_{l}\\
	&\hspace{40pt}\times H_{l}^{(n),(\mu)}(q^{-1},t^{-1})\Biggr]\bra{0}\Gamma_{n}(X)_{+}.
\end{align*}
Let us set
\begin{equation*}
	\mfrak{X}:=\sum_{k=1}^{n}(1-tq^{\mu_{k}})\prod_{\substack{i=1\\ i\neq k}}^{n}\frac{x_{k}-tq^{\mu_{i}}x_{i}}{x_{k}-x_{i}}T_{q^{-1},x_{k}}H_{l}^{(n),(\mu+\epsilon_{k})}(q^{-1},t^{-1}).
\end{equation*}
Then, using Eqs. (\ref{eq:assist1}) and (\ref{eq:assist2}), we can verify that
\begin{align*}
	\mfrak{X}
	&=t^{n}q^{|\mu|}\sum_{\substack{\nu\in(\mbb{Z}_{\ge 0})^{n}\\ |\nu|=l}}\sum_{k=1}^{n}
				(q^{-\nu_{k}-1}-1)\prod_{\substack{i=1 \\ i\neq k}}^{n}\frac{x_{k}-q^{-\nu_{i}}x_{i}}{x_{k}-x_{i}}h_{\nu+\epsilon_{k}}^{(\mu)}(q^{-1},t^{-1})
				T_{q^{-1},x_{k}}\prod_{i=1}^{n}T_{q^{-1},x_{i}}^{\nu_{i}} \\
	&=t^{n}q^{|\mu|}\sum_{\substack{\nu\in(\mbb{Z}_{\ge 0})^{n}\\ |\nu|=l+1}}\left(\sum_{k=1}^{n}
				(q^{-\nu_{k}}-1)\prod_{\substack{i=1 \\ i\neq k}}^{n}\frac{x_{k}-q^{-\nu_{i}}x_{i}}{x_{k}-x_{i}}\right)h_{\nu}^{(\mu)}(q^{-1},t^{-1})
				\prod_{i=1}^{n}T_{q^{-1},x_{i}}^{\nu_{i}}.		
\end{align*}
Here we use Lemma \ref{lem:partial_fraction} to obtain
\begin{equation*}
	\mfrak{X}=t^{n}q^{|\mu|}(q^{-l-1}-1)H_{l+1}^{(n),(\mu)}(q^{-1},t^{-1}).
\end{equation*}
Therefore
\begin{align*}
	\mfrak{G}_{r}^{(\mu)}(q^{-1},t^{-1})
	&=\Biggl[t^{(r-1)n}q^{(r-1)(|\mu|+1)}\sum_{l=0}^{r-1}(-1)^{l}q^{\binom{l}{2}}q^{l(r-l+1)}(q^{-r+l};q^{-1})_{l} \\
	&\hspace{40pt}\times t^{n}q^{|\mu|}(q^{-l-1}-1)H_{l+1}^{(n),(\mu)}(q^{-1},t^{-1}) \\
	&\hspace{20pt}+t^{rn}q^{r|\mu|}\sum_{l=0}^{r-1}(-1)^{l}q^{\binom{l}{2}}q^{l(r-l+1)}(q^{-r+l};q^{-1})_{l}\\
	&\hspace{40pt}\times H_{l}^{(n),(\mu)}(q^{-1},t^{-1})\Biggr]\bra{0}\Gamma_{n}(X)_{+}.
\end{align*}
It can be checked that this coincides with
\begin{align*}
	t^{rn}q^{r|\mu|}\sum_{l=0}^{r}(-1)^{l}q^{\binom{l}{2}}q^{l(r-l)}(q^{-r+l-1};q^{-1})_{l}H_{l}^{(n),(\mu)}(q^{-1},t^{-1})\bra{0}\Gamma_{n}(X)_{+}.
\end{align*}
Then the proof is complete.
\end{proof}

\bibliographystyle{alpha}
\bibliography{mac_process}

\end{document}